\documentclass[onefignum,onetabnum]{siamltex1213}

\usepackage{epsfig}
\usepackage{amsfonts}
\usepackage{amsmath}
\usepackage{algorithm}
\usepackage{algorithmic}
\usepackage{multirow}
\usepackage{graphicx}
\usepackage{subfig}

\usepackage{cite}
\graphicspath{{./}{./figures/}}

\newcommand{\cN}{{\cal N}}
\newcommand{\cO}{{\cal O}}
\newcommand{\sV}{{\cal V}}
\newcommand{\sQ}{{\cal Q}}
\newcommand{\sK}{{\cal K}}
\newcommand{\sM}{{\cal M}}
\newcommand{\sS}{{\cal S}}
\newcommand{\uu}{{\underline{u}{}}}
\newcommand{\ulambda}{{\underline{\lambda}{}}}
\newcommand{\uA}{{\underline{A}{}}}
\newcommand{\uB}{{\underline{B}{}}}
\newcommand{\uc}{{\underline{c}{}}}
\newcommand{\uf}{{\underline{f}{}}}
\newcommand{\ug}{{\underline{g}{}}}
\newcommand{\us}{{\underline{s}{}}}
\newcommand{\uv}{{\underline{v}{}}}
\newcommand{\uuu}{\underline{\underline{u}{}}}
\newcommand{\uus}{\underline{\underline{s}{}}}
\newcommand{\uulambda}{\underline{\underline{\lambda}{}}}
\newcommand{\uuv}{\underline{\underline{v}{}}}
\newcommand{\uuA}{\underline{\underline{A}{}}}
\newcommand{\uuB}{\underline{\underline{B}{}}}
\newcommand{\uuf}{\underline{\underline{f}{}}}
\newcommand{\uug}{\underline{\underline{g}{}}}
\newcommand{\usK}{\underline{{\sK}{}}}
\newcommand{\usS}{\underline{{\sS}{}}}
\newcommand{\tA}{{\tilde{A}{}}}
\newcommand{\tf}{{\tilde{f}{}}}
\newcommand{\tQ}{{\tilde{Q}{}}}
\newcommand{\tTheta}{{\tilde{\Theta}{}}}
\newcommand{\utA}{\underline{\tilde{A}{}}}
\newcommand{\uutA}{\underline{\underline{\tilde{A}}}{}}
\newcommand{\utf}{\underline{\tilde{f}{}}}
\newcommand{\uutf}{\underline{\underline{\tilde{f}}}{}}
\newcommand{\uxi}{{\underline{\xi}{}}}

\newcommand{\uvarphi}{{\underline{\varphi}{}}}
\newcommand{\upsi}{{\underline{\psi}{}}}
\newcommand{\usigma}{{\underline{\sigma}{}}}
\newcommand{\uzeta}{{\underline{\zeta}{}}}

\newcommand{\ang}[1]{\langle #1\rangle}

\newcommand{\RcNV}{\mathbb{R}^{\cN_\sV}}
\newcommand{\RcNQ}{\mathbb{R}^{\cN_\sQ}}
\newcommand{\RcnV}{\mathbb{R}^{n_\sV}}
\newcommand{\RcnQ}{\mathbb{R}^{n_\sQ}}
\newcommand{\pr}{{\rm pr}}
\newcommand{\du}{{\rm du}}
\newcommand{\Ppr}{{\bf P$^{\rm pr}$}}
\newcommand{\Pprn}{{\bf P{}$^{\rm pr}_n$}}
\newcommand{\uPpr}{\underline{\bf P}${}^{\rm pr}$}

\newcommand{\uuPprn}{\underline{{\underline{\bf P}}}{}$^{\rm pr}_n$}
\newcommand{\Pdu}{{\bf P$^{\rm du}$}}

\newcommand{\uPdu}{\underline{\bf P}${}^{\rm du}$}
\newcommand{\uPdun}{\underline{\bf P}{}$^{\rm du}_n$}
\newcommand{\uuPdun}{{\underline{\underline{\bf P}}}{}$^{\rm du}_n$}

\newcommand{\cD} {{\cal D}}

\def\problem[#1]{{\par{\sc Problem {\bf \rm #1}.}\ignorespaces}}

\title{A Duality Approach to Error Estimation for Variational Inequalities
\thanks{This work was
supported by the Deutsche Forschungsgemeinschaft (German Research Foundation) through grant GSC 111.}} 

\author{
Z.~Zhang
\and E.~Bader
\and K.~Veroy\footnotemark[2]
}

\begin{document}
\maketitle
\slugger{mms}{xxxx}{xx}{x}{x--x}%

\footnotetext[2]{Aachen Institute for Advanced Study in Computational Engineering Science (AICES) Graduate School, Schinkelstrasse 2, 52056 Aachen, Germany}
\begin{abstract}
Motivated by problems in contact mechanics, we propose a duality approach for computing approximations and associated a posteriori error bounds to solutions of variational inequalities of the first kind.  The proposed approach improves upon existing methods introduced in the context of the reduced basis method in two ways.  First, it provides sharp {\it a posteriori} error bounds which mimic the rate of convergence of the RB approximation.  Second, it enables a {\it full} offline-online computational decomposition in which the online cost is completely independent of the dimension of the original (high-dimensional) problem.  Numerical results comparing the performance of the proposed and existing approaches illustrate the superiority of the duality approach in cases where the dimension of the full problem is high.  
\end{abstract}


\begin{AMS}{35J86, 65K15, 65N15, 90C33} \end{AMS}

\pagestyle{myheadings}
\thispagestyle{plain}
\markboth{Z.~ZHANG, E.~BADER, AND K.~VEROY}{DUALITY APPROACH TO ERROR ESTIMATION FOR VARIATIONAL INEQUALITIES\ \ }

\section{Introduction}
\label{sec:intro}

We present an efficient model order reduction method for parametrized elliptic variational inequalities of the first kind.  Motivated by numerous engineering applications that involve contact between elastic bodies, we develop a primal-dual reduced basis (RB) approach to constructing online-inexpensive yet certified reduced order models.  Such models find application in the real-time or many query context of PDE-constrained optimisation, control, or parameter estimation.   

The use of the certified RB approach for  variational inequalities (VIs) has been explored in \cite{HSW2012} for elliptic problems, and in \cite{Burkovska2014},  \cite{Glas2013}, and  \cite{Haasdonk2012},  for parabolic problems.  However, we foresee that two aspects of the approach presented in \cite{HSW2012} (upon which \cite{Burkovska2014},  \cite{Glas2013}, and  \cite{Haasdonk2012} are based) will likely cause difficulties when considering  high-dimensional problems in contact mechanics.  In particular, the approach in \cite{HSW2012} ({\it i\/}) provides only a {\it partial} offline/online decomposition, i.e., the online  cost to compute the {\it a posteriori} error bounds depends on the dimension of the finite element solution; and ({\it ii\/}) the convergence rate of the {\it a posteriori} error bounds does not mimic the actual convergence rate of the approximation.  

In this work, we develop a certified RB method that provides sharper {\it and} inexpensive {\it a posteriori} error bounds.  In particular, our primal-dual approach not only ({\it i\/}) provides sharp error bounds that mimic the  convergence rate of the RB approximation, but also  ({\it ii}) does so at an online cost that is independent of the high dimension of the original problem.  We illustrate these claims using two model problems.  

The paper is organized as follows. In Sec.~2 we present different abstract formulations of our problem and state known theoretical results regarding existence and uniqueness of solutions.  In Sec.~3 we summarise the approximation and {\it a posteriori} error estimation approach presented in \cite{HSW2012}; in this paper we shall refer to this as the {\it primal-only} approach.  In Sec.~4 we present our proposed {\it primal-dual} approach, in which an additional problem expressed in terms of the slack variable is introduced.  To emphasize the generality of the approach, as well as for purposes of clarity, we delay until Sec.~5 the introduction of the parametrized problem and the application of the above-mentioned approaches to the RB method. The model problems and corresponding numerical results are then presented in Sec.~6. 

\section{Problem Statement} 
\label{sec:ProbStatement}

We consider several different formulations of our problem: a minimization statement, a general (standard) variational inequality, a mixed formulation, and a mixed complementarity problem.  Since, in practice, we typically consider finite dimensional approximations, we shall also present the corresponding algebraic formulation.

\subsection{Preliminaries} 

Let $\sV$ and $\sQ \subset L^2(\Omega)$ be two separable Hilbert spaces, with inner products $(\cdot,\cdot)_{\sV}$, $(\cdot,\cdot)_{\sQ},$ and associated norms $\|\cdot\|_{\sV} = \sqrt{(\cdot,\cdot)_{\sV}}$, $\|\cdot\|_{\sQ}=\sqrt{(\cdot,\cdot)_{\sQ}}$, respectively.  Here,  $\Omega \subset \mathbb{R}^d, d=1,2,3,$ is a bounded Lipschitz domain.   The corresponding dual spaces are denoted by $\sV'$ and $\sQ'$; we further denote a general duality pairing as $\langle \cdot, \cdot \rangle$.

Let $a(\cdot,\cdot):\sV \times \sV \to \mathbb{R}$ be a continuous, coercive bilinear form, and let $A :\sV \rightarrow \sV'$ be the induced linear map,
$\ang{Av,w} = a(v,w), \ \forall \;v,w \in \sV.$
Note that $a(\cdot,\cdot)$ is not necessarily symmetric.  
We then define the continuity and coercivity constants as 
\begin{equation}
\label{eq:alpha-gamma}
\gamma  \equiv \sup_{w\in \sV} \sup_{v\in \sV} 
\frac{\ang{Aw,v}}{\|w\|_{\sV}\|v\|_{\sV}} < \infty, 
\qquad \alpha \equiv \inf_{v\in \sV} \frac{\ang{Av,v}}{\|v\|^2_{\sV}} > 0.
\end{equation}
We further introduce a linear functional $f \in \sV'$. 

\subsection{Variational Inequality}
 
The study of variational inequalities has its origins in \cite{Fichera1963}.  We consider here VIs of the {\it first} kind, i.e., VIs that are posed on convex subsets.  We thus let $\sK$ be a non-empty closed convex subset of $\sV$ and state the abstract form of a VI of the first kind.  
\begin{problem}[\bf{A1}]  {\it Find $u \in \sK$ such that
\begin{equation} 
\label{eq:vi}
\ang{Au,v -u} \geq \ang{f,v-u}, \quad\forall\;v\in \sK.
\end{equation}}
\noindent
In cases in which the bilinear form $a(\cdot,\cdot)$ is symmetric, the variational inequality (\ref{eq:vi}) is equivalent to the constrained minimization problem
$$u = \arg \min_{v \in \sK} \textstyle\frac12 \ang{Av,v} - \ang{f,v}.$$
For more details on the equivalence of the two formulations in the symmetric case, we refer the reader to \cite{Kikuchi1988}.  We now (re-)state a well-known result on the existence and uniqueness of a solution to {\bf A1}.

\end{problem}

\begin{theorem}
\label{thm:LS67}
{\bf [Lions \& Stampacchia, 1967]} Let $A:\sV \to \sV'$ satisfy (\ref{eq:alpha-gamma}) and let $\sK$ be a non-empty closed convex set of $\sV$. There exists a unique solution of Problem~{\bf A1}. Furthermore, the map $f \to u$ (generally nonlinear) is continuous from $\sV'$ into $\sV$,
\end{theorem}
\begin{proof}
We refer the reader to \cite{Lions1967} for the proof.
\end{proof}

\subsection{Mixed Formulation} 

We now introduce a proper positive cone $\sM$ of the space $\sQ$ and a corresponding positive cone $\sM'$ in the dual space $\sQ'$ defined as 
\begin{subequations}
\begin{eqnarray}
\sM & := & \{ \ q \in Q\phantom{'} \; | \; q \geq 0 \mbox{ a.e. in } \Omega\ \}, \\
\label{eq:cone}
\sM' & := & \{ \ \zeta \in \sQ' \; | \; \langle \zeta, q \rangle \geq 0, \quad \forall \ q \in \sM\ \}.
\end{eqnarray}
\end{subequations}
For more details on the above concepts, we refer the reader to \cite{Glowinski1984} and \cite{Kikuchi1988}. 

We also assume that the convex set $\sK$ is given by 
\begin{equation}
\sK := \left\{ v \in V \; | \; \ang{Bv,q} \leq \ang{g,q},\ \forall \; q \in \sM \right\},
\end{equation}
where $B: \sV \to \sQ' $ is the induced linear map of a continuous bilinear form $b(\cdot,\cdot):\sV \times \sQ \to \mathbb{R}$,
$
\ang{Bv,q} = b(v,q), \quad \forall \;v \in \sV, \ q \in \sQ; 
$
and $g$ is a bounded linear functional, $g \in \sQ'$.  In Sec.~\ref{sec:duality}, we shall also assume that $B$ is bijective, so that $B^{-1}$ is well-defined.  

We now consider the following mixed formulation of our variational inequality:
\begin{problem}[\bf{A2}]
\label{pr:mixed}
{\it Find $(u,\lambda) \in \sV \times \sM$ such that
\begin{subequations}
\label{eq:spp}
\begin{alignat} {6}
\label{eq:spp1}
\ang{Au,v}  + \ang{Bv, \lambda} & = & \ &  \ang{f,v},  & & \quad\forall\;v\in \sV \\
\label{eq:spp2}
\ang{Bu,q-\lambda} & \leq & \ & \ang{g,q-\lambda}, & & \quad\forall\;q\in \sM
\end{alignat}
\end{subequations}}
~\\[-2ex]
\end{problem}
\noindent We then summarise some results on the existence, uniqueness, and boundedness of the solution to {\bf A2}.
{\begin{theorem}
\label{thm:BHR78}
Let $A:\sV \to \sV'$ satisfy (\ref{eq:alpha-gamma}) and let $\sK$ be a non-emptyclosed convex set of $\sV$.  Suppose further that there exists a constant $\beta_0>0$ such that
\begin{equation} 
\label{eq:beta_b}
\beta \equiv \inf_{q\in \sQ}  \sup_{v\in \sV} \frac{\ang{Bv,q}}{\|q\|_{\sQ}\|v\|_{\sV}} \geq \beta_0 > 0.
\end{equation}
Then {\bf A2} has a unique solution.  Furthermore, if $(u,\lambda)$ solves {\bf A2}, then $u$ solves {\bf A1}.  
\end{theorem}
\begin{proof}
This statement is essentially a slight modification of Thm.~\ref{thm:LS67} of  {\cite{Brezzi1978}} which states that if such a $\beta_0$ exists, then ({\it i\/}) problems (\ref{eq:vi}) and (\ref{eq:spp}) have at most one solution; ({\it ii\/}) if either problem has a solution, then they both have solutions; and ({\it iii\/}) if $(u,\lambda)$ solves {\bf A2}, then $u$ solves {\bf A1}.  We omit the proof here, and refer the reader to {\cite{Brezzi1978}} for details. The result thus directly follows from Theorem~\ref{thm:LS67}. \quad
\end{proof}
}

\subsection{Mixed Complementarity Problem}  

In the discrete setting, it is well-known that (\ref{eq:vi}) is equivalent to a mixed complementarity problem (see \cite{Karush1939} \& \cite{KT1951} or, e.g., \cite{Facchinei2003}).  We present here a simple extension of this result to our function space setting.

We begin by deriving the mixed complementarity problem (also known as the {\it generalized Karush-Kuhn-Tucker conditions}) corresponding to {\bf A2}:
\begin{problem}[\bf{ P$^{\pr}$}]
\label{pr:micp}
{\it Find $(u,\lambda) \in \sV \times \sQ$ such that}
\begin{subequations}
\label{eq:gkkt}
\begin{alignat}{7}
\label{eq:gkkt1}
\ang{Au,v}  + \ang{Bv, \lambda} & = & \ &  \ang{f,v},  & & \quad\forall\;v\in \sV \\
\label{eq:gkkt2}
\ang{g,q} - \ang{Bu,q} & \geq & \ & 0, & &\quad\forall\;q\in \sM \\
\label{eq:gkkt3}
\lambda & \geq & \ & 0, & & \\
\ang{g-Bu,\lambda} & = & \ & 0.
\label{eq:gkkt4}
\end{alignat}
\end{subequations}

\end{problem}

\begin{lemma} Under the assumptions of Thm.~\ref{thm:BHR78}, there exists a unique solution $(u,\lambda) \in \sV \times \sQ$ to \Ppr.  Furthermore, the pair $(u,\lambda)$ solves {\bf A2} if and only if it solves  {\bf{P$^{\pr}$}}. \label{le:P1eqPpr}
\end{lemma}
\begin{proof}
We first note that if the second statement holds, then the first statement follows directly from Theorem \ref{thm:BHR78}. We thus need only to prove the second statement.  Clearly, (\ref{eq:spp1}) and (\ref{eq:gkkt1}) are equivalent, and we note  from the definition of $\sM$ that (\ref{eq:gkkt3}) is true if and only if $\lambda \in \sM$.  It thus remains to prove that ({\it i\/}) (\ref{eq:spp}) implies (\ref{eq:gkkt2}) and (\ref{eq:gkkt4}), and that  ({\it ii\/}) (\ref{eq:gkkt}) implies (\ref{eq:spp2}).

To prove ({\it i\/}), we note that taking $q = q'+\lambda$ in (\ref{eq:spp2}) yields (\ref{eq:gkkt2}).  Furthermore, (\ref{eq:gkkt4}) follows from taking $q = 2\lambda$ and $q = 0$ in (\ref{eq:spp2}).  To prove ({\it ii\/}), we subtract (\ref{eq:gkkt4}) from (\ref{eq:gkkt2}) to obtain  (\ref{eq:spp2}).  This completes the proof.  \qquad 
\end{proof}

Note that in anticipation of the dual problem to be introduced in later sections, we refer to (\ref{eq:gkkt}) as our {\it primal} problem, {\bf P$^{\pr}$}. 

\subsection{Algebraic Problem Statement}
\label{sec:alg_prob}
Although the problems stated above may be infinite-dimensional, they may, in fact, also represent a ``truth" approximation --- a finite-dimensional, high-fidelity approximation to the infinite-dimensional problem.  

We assume that $\sV$ (respectively, $\sQ$) is finite-dimensional with dimension $\cN_\sV$  ($\cN_\sQ$) and can be represented in terms of basis functions $\Phi_i$ ($\Psi_j$):
\begin{equation}
\label{eq:fe_basis}
\sV = {\rm span}\{ \Phi_i, \ 1 \leq i \leq \cN_\sV\}, \qquad \sQ = {\rm span}\{ \Psi_j, \ 1 \leq j \leq \cN_\sQ\}.
\end{equation}
We further assume that some non-negative basis functions $\Psi_j$ are chosen such that the convex cone $\sM \subset \sQ$ is  given by:
$$ \sM = {\rm span}_+\{ \Psi_j \} = \{ q \in \sQ \; | \; q = \textstyle\sum\limits_{j=1}^{\cN_\sQ} \underline{q}_j \Psi_j \ {\rm and} \ \underline{q} \in \mathbb{R}_+^{\cN_\sQ} \},$$
where $\mathbb{R}_+:= \{ c \in \mathbb{R} \,| \, c \geq 0\}$.  Here, the single underline signifies a vector of coefficients corresponding to the appropriate basis in ({\ref{eq:fe_basis}}).  We then let $\uu \in \mathbb{R}^{\cN_\sV}$ and $\ulambda \in \mathbb{R}_+^{\cN_\sQ}$ be the vector of coefficients of $u$ and $\lambda$:
\begin{equation}
u = \textstyle\sum\limits_{i=1}^{\cN_\sV} \uu_i \Phi_i, \qquad \lambda = \textstyle\sum\limits_{j=1}^{\cN_\sQ} \ulambda_j \Psi_j.
\nonumber
\end{equation}
The coefficients $(\uu,\ulambda)$ are then obtained by solving
\begin{problem}[\uPpr]
Find $(\uu,\ulambda) \in \mathbb{R}^{\cN_\sV} \times \mathbb{R}^{\cN_\sQ}$
\begin{subequations}
\label{eq:kkt}
\begin{eqnarray}
\label{eq:kkt1}
\uA \,\uu + \uB\, \ulambda & = &  \uf \\
\label{eq:kkt2}
\ug - \uB\,\uu & \geq & 0 \\
\label{eq:kkt3}
\ulambda & \geq & 0 \\
(\ug - \uB\,\uu)^T {\ulambda} & = & 0,
\label{eq:kkt4}
\end{eqnarray}
\end{subequations}
where $\uA_{ij} = \ang{A \Phi_j, \Phi_i}$,  $\uB_{kj} = \ang{B \Phi_j, \Psi_k}$, $\uf_i = \ang{f,\Phi_i}$, and $\ug_k = \ang{g,\Psi_k}$ for $1 \leq i,j \leq \cN_\sV$, and $1 \leq k \leq \cN_\sQ$.  
\end{problem}

We now consider approximations to the ``truth", as well as methods for {\it a posteriori} estimation of the error with respect to the ``truth".  We consider two approaches: a primal-only approach in Sec.~\ref{sec:primal}, and a duality approach in Sec.~\ref{sec:duality}.  

\section{Approximation and Error Estimation: A Primal-Only Approach}
\label{sec:primal}
The approach described in this section was first introduced in \cite{HSW2012} in the context of the reduced basis method for elliptic VIs and subsequently applied to parabolic VIs in \cite{Burkovska2014}, \cite{Glas2013}, and \cite{Haasdonk2012}.  Since we shall require key elements of this earlier work in our proposed approach in Sec.~\ref{sec:duality}, we summarise here the main aspects of the methods in \cite{HSW2012}.  To emphasize the generality of our proposed approach and also for the sake of clarity, we postpone the introduction of parameters to Sec.~\ref{sec:rb}.

\subsection{Approximation}
\label{sec:primalapprox}

Following \cite{HSW2012}, we let the approximation spaces $\sV_n \subset \sV$ and $\sQ_n \subset \sQ$ be given by
\begin{equation}
\label{eq:basis_n}
\sV_n = {\rm span} \{ \varphi_i \in \sV, \ 1 \leq i \leq n_\sV\},  \qquad \sQ_n = {\rm span} \{ \psi_j \in \sM, \ 1 \leq j \leq n_\sQ \},
\end{equation}  
where the basis functions $\varphi_i$, $\psi_j$ are assumed to be linearly independent.  
We further let $\sM_n \subset \sQ_n$ be the closed convex cone
$$ \sM_n = \{ q \in \sQ_n \, | \, q = \textstyle\sum\limits_{j=1}^{n_\sQ} \underline{\underline{q}}_j \psi_j \ {\rm and} \ \underline{\underline{q}} \in \mathbb{R}_+^{n_\sQ} \},$$
and define 
\begin{equation}
\label{eq:K_n}
\sK_n := \left\{ v_n \in \sV_n \; | \; \ang{Bv_n,q_n} \leq \ang{g,q_n},\ \forall \; q_n \in \sM_n \right\}.
\end{equation}
Note that in general, $\sK_n \not\subset \sK$ since the inequality in \eqref{eq:basis_n} holds only in $\sM_n$, not $\sM$.

We now consider the following approximation to Problem~{\bf P$^{\pr}$}.  \begin{problem}[{\bf P$^\pr_n$}]
\label{pr:micp-app}
Find $(u_n^\pr,\lambda_n^\pr) \in \sV_n \times \sQ_n$ such that
\begin{subequations}
\begin{alignat}{7}
\label{eq:gkkt-app-1}
\ang{Au_n^\pr,v_n}  + \ang{Bv_n, \lambda_n^\pr} & = & \ &  \ang{f,v_n},  & & \quad\forall\;v_n\in \sV_n \\
\label{eq:gkkt-app-2}
\ang{g,q_n} - \ang{Bu_n^\pr,q_n} & \geq & \ & 0, & &\quad\forall\;q_n\in \sM_n \\
\label{eq:gkkt-app-3}
\lambda_n^\pr & \geq & \ & 0, & &  \\
\ang{g-Bu_n^\pr,\lambda_n^\pr} & = & \ & 0.
\label{eq:gkkt-app-4}
\end{alignat}
\end{subequations}
\end{problem}
\begin{corollary}
\label{cor:bhr}
Suppose there exists a constant $\beta_0>0$ such that 
\begin{equation} 
\label{eq:beta_b_n}
\beta_n \equiv \inf_{q_n\in \sQ_n}  \sup_{v_n\in \sV_n} \frac{b(v_n,q_n)}{\|q_n\|_{\sQ}\|v_n\|_{\sV_n}} \geq \beta_0 > 0, 
\end{equation}
for $n \in {\mathbb N}$.  Then there exists a unique solution $(u_n,\lambda_n)$ to {\bf P$^\pr_n$}.\end{corollary}
\begin{proof}
The result directly follows from Thm.~\ref{thm:BHR78}; see also \cite{HSW2012}. \qquad  
\end{proof}

As in Sec.~\ref{sec:ProbStatement}, we now derive the algebraic equations corresponding to {\bf P$^\pr_n$}.  Following the notation in Sec.~\ref{sec:alg_prob}, we denote by ${\uvarphi}{}_i \in \mathbb{R}^{\cN_\sV}$ and $\upsi{}_j \in \mathbb{R}^{\cN_\sQ}$ the coefficients of $\varphi_i$ and $\psi_j$ corresponding to the  ``truth"  bases in (\ref{eq:fe_basis}).  Furthermore, we shall from here on use a double underline to signify a vector of coefficients corresponding to the approximation bases in (\ref{eq:basis_n}).  In particular, any $v_n \in \sV_n$ can then be written as
\begin{equation}
v_n \, = \, \textstyle\sum\limits_{i=1}^{n_\sV} \uuv{}_{n,i} \varphi{}_i \, = \, \textstyle\sum\limits_{i=1}^{n_\sV} \uuv{}_{n,i} \bigg(\textstyle\sum\limits_{j=1}^{\cN_\sV}{\uvarphi}{}_{i,j} \Phi{}_j\bigg).
\label{eq:different_bases}
\end{equation}
We note from (\ref{eq:different_bases}) that we can express $\uv_n$  in terms of $\uuv_n$ as
$ \uv_n = \textstyle\sum_{i=1}^{n_\sV}\uuv_{n,i}\underline{\varphi}_{i}.$
We can now readily derive the algebraic formulation of \Pprn.
  
\begin{problem}[\uuPprn]
\label{pr:micp-app-alg}
Find $(\uuu_n^\pr,\uulambda_n^\pr) \in \mathbb{R}^{n_\sV} \times \mathbb{R}^{n_\sQ}$ such that
\begin{subequations}
\begin{alignat}{7}
\label{eq:rkkt-pr-1}
\uuA{}_n \,\uuu{}^\pr_n + \uuB{}_n\, \uulambda{}^\pr_n & = &  \uuf{}_n \\
\label{eq:rkkt-pr-2}
\uug{}_n - \uuB{}_n\,\uuu{}^\pr_n & \geq & 0 \\
\label{eq:rkkt-pr-3}
\uulambda{}^\pr_n & \geq & 0 \\
(\uug{}_n - \uuB{}_n\,\uuu{}^\pr_n)^T {\uulambda{}^\pr_n} & = & 0,
\label{eq:rkkt-pr-4}
\end{alignat}
\end{subequations}
where ${\uuA}{}_{n,ij} = \uvarphi{}_i^T \uA\, \uvarphi_j$,  ${\uuB}{}_{n,kj} = \upsi{}_k^T \uB\, \uvarphi_j$, ${\uuf}{}_i = \uvarphi{}^T_i{\uf}$, and ${\uug}{}_k = \upsi{}^T_k{\ug}$ for $i,j = 1,\dots,n_{\sV}$ and $j=1,\dots,n_{\sQ}$.  
\end{problem}

In this paper, we are concerned with obtaining upper bounds of the error in our approximation $(u_n,\lambda_n)$ with respect to our truth solution $(u,\lambda)$.  Later we use these upper bounds in the reduced basis setting, but these error estimates are, in fact, quite general.  To motivate the methods that we propose in Sec.~\ref{sec:duality}, we first review the relevant results in the literature. 

\subsection{A Posteriori Error Estimation}
\label{sec:pr-a-post-error}

In this section, we summarise results initially presented in \cite{HSW2012} for elliptic variational inequalities and subsequently applied and extended to parabolic variational inequalities in \cite{Haasdonk2012},  \cite{Glas2013} and \cite{Burkovska2014}.  Following \cite{HSW2012}, we let $r_{\rm e} \in \sV'$ and $r_{\rm i} \in \sQ'$ be given by
\begin{subequations}
\begin{alignat}{5}
\label{eq:residual-e}
r_{\rm e}(v) & := &&\ \ang{f,v} - \ang{Au_n^\pr,v} - \ang{B v, \lambda_n^\pr}, & \qquad \forall \; v \in \sV, \\
\label{eq:residual-i}
r_{\rm i}(q) & := && \ \ang{Bu_n^\pr,q} - \ang{g,q}, & \qquad \forall \; q \in \sQ.
\end{alignat}
\end{subequations}
In \cite{HSW2012}, $r_{\rm e}(\cdot)$ and $r_{\rm i}(\cdot)$ are denoted the ``equality and inequality residuals", respectively.  We let ${e}_{\rm i} \in \sQ$ and $\tilde{e}_{\rm i} \in \sQ$ be the Riesz representation of, respectively, the inequality residual and the ``detailed inequality functional"
\begin{subequations}
\begin{alignat}{5}
\label{eq:riesz-e}
({e}_{\rm i},q)_\sQ & = &&\ \ang{Bu_n^\pr,q} - \ang{g,q} = r_{\rm i}(q), & \qquad  \forall \; q \in \sQ, \\
\label{eq:riesz-etilde}
(\tilde{e}_{\rm i},q)_\sQ & = &&\ \ang{Bu,q} - \ang{g,q}, & \qquad  \forall \; q \in \sQ.
\end{alignat}
\end{subequations}
The development in \cite{HSW2012} further requires a projection $\Pi: \sQ \to \sM$ that is orthogonal with respect to a scalar product $(\cdot,\cdot)_\Pi$ on $\sQ$ endowed with the induced norm $\| q \|_\Pi := \sqrt{(q,q)_\Pi}$.  The (generally nonlinear) projection $\Pi$ is then assumed to satisfy
\begin{subequations}
\begin{alignat}{5}
\label{eq:proj-1}
(q-\Pi(q),\eta)_\sQ & \leq &\ & 0, &&\qquad  \forall q \, \in \sQ, \ \forall \, \eta \in \sM, \\
\label{eq:proj-2}
\Pi(\tilde{e}_{\rm i}) & = && 0, \\
\label{eq:proj-3}
(q,\tilde{e}_{\rm i})_\Pi & \leq &\ & 0, &&\qquad  \forall \, q \in \sM.
\end{alignat}
\end{subequations}
We refer the reader to \cite{HSW2012} for further details on the choice of $\Pi$.  

We now state the main results of \cite{HSW2012}.  To motivate the methods that we propose in Sec.~\ref{sec:duality} (and to properly analyze the differences between the two methods), we include the detailed proofs here.  To begin, we shall need
\begin{lemma} 
\label{lem:dual-err}
The error in the approximation for the KKT  multiplier $\lambda$ can be bounded in terms of the error in the approximation for the primal variable $u$: 
\begin{equation} 
\label{eq:dual-err}
\| \lambda - \lambda_n^\pr \|_\sQ \leq \frac{1}{\beta} \big(\|r_{\rm e}\|_{\sV'} + \gamma \|u - u_n^\pr\|_\sV\big).
\end{equation}
\end{lemma}
\begin{proof}
This statement is a direct application of the result (and proof) of Prop. 1.3 (Sec.~II.1) in \cite{BF1991}.  We thus refer the reader to \cite{BF1991} for more details.   \qquad
\end{proof}

We then derive a posteriori error estimators in
\begin{proposition} 
\label{prop:HSW}
{\bf [Haasdonk, Salomon \& Wohlmuth, 2012]}  We define the residual estimators
\label{prop:error-bound}
\begin{equation}
\label{eq:res-est}
\delta_0 := \| r_{\rm e}\|_{\sV'}, \qquad \delta_1:= \| \Pi({e}_{\rm i}) \|_\sQ, \qquad \delta_2 := ({\lambda_n,\Pi({e}_{\rm i}))_\sQ}
\end{equation}
and the constants
\begin{equation}
\label{eq:pr-constants-hat}
\hat{c}_1 := \frac{1}{2\alpha} \left( \delta_0 + \frac{\gamma \delta_1  }{\beta}\right),\qquad
\hat{c}_2 :=  \frac{1}{\alpha} \left( \frac{\delta_0\delta_1}{\beta}   + \delta_2\right).
\end{equation}
The errors can then be bounded by 
\begin{subequations}
\begin{alignat}{8}
\label{eq:pron-errbnd-pr}
\|u-u_n^\pr\|_\sV & \leq & \ & \hat{\Delta}_u^{\pr} & \ & := &  \ & \hat{c}_1 + \sqrt{\hat{c}_1^2 + \hat{c}_2} \quad \\
\label{eq:pron-errbnd-du}
\|\lambda-\lambda_n^\pr\|_\sQ & \leq & \ & \hat{\Delta}_\lambda^{\pr} & \ & := &  \ & \frac{1}{\beta} \left( \delta_0 + \gamma \hat{\Delta}_u^{\pr} \right) 
\end{alignat}
\end{subequations}
\end{proposition}
\begin{proof}
For conciseness, we omit in this proof the superscript ``\pr."  We thus emphasize that here, $(u_n,\lambda_n)$ refers to our {\it primal} approximation $(u_n^\pr,\lambda_n^\pr)$.  

To begin, we note that the result (\ref{eq:pron-errbnd-du}) follows directly from Lemma~\ref{lem:dual-err} and (\ref{eq:pron-errbnd-pr}).  It thus remains to prove (\ref{eq:pron-errbnd-pr}).  From (\ref{eq:alpha-gamma}), (\ref{eq:residual-e}), and  (\ref{eq:res-est}), we have
\begin{eqnarray}
\nonumber
\alpha \| u-u_n\|_\sV^2 & \; \leq \; & \ang{A(u-u_n),u-u_n} \\
\nonumber
& = & r_{\rm e}(u-u_n) - \ang{B(u-u_n),\lambda - \lambda_n} \\
\label{eq:pron-1}
& \leq & \delta_0 \| u-u_n \|_\sV  + \ang{Bu,\lambda_n - \lambda} -\ang{Bu_n,\lambda_n - \lambda}.
\end{eqnarray}
Using (\ref{eq:spp2}), (\ref{eq:residual-i}),  (\ref{eq:gkkt-app-2}), (\ref{eq:riesz-e}), (\ref{eq:proj-1}), and (\ref{eq:res-est}), we have
\begin{subequations}
\begin{eqnarray}
\nonumber
 \ang{Bu,\lambda_n - \lambda} -\ang{Bu_n,\lambda_n - \lambda}  & \; \leq \; & \ang{g,\lambda_n - \lambda} - \Big(\ang{r_{\rm i},\lambda_n - \lambda} + \ang{g,\lambda_n - \lambda}\Big) \qquad \\
& \; \leq \; & {r_{\rm i}}(\lambda) 
\label{eq:almostorth} \\
\nonumber
& \leq & (\lambda,e_{\rm i}-\Pi(e_{\rm i}))_\sQ + (\lambda,\Pi(e_{\rm i}))_\sQ\\
\nonumber
& \leq & (\lambda,\Pi(e_{\rm i}))_\sQ \\ 
\nonumber
& = & (\lambda - \lambda_n,\Pi(e_{\rm i}))_\sQ + (\lambda_n,\Pi(e_{\rm i}))_\sQ \\
\label{eq:cauchy-schwarz} 
& \leq & \delta_1 \|\lambda - \lambda_n\|_\sQ + \delta_2.
\end{eqnarray}
\end{subequations}
Substituting into (\ref{eq:pron-1}) and applying (\ref{eq:dual-err}) and (\ref{eq:pr-constants-hat}), it  follows that
$$ \| u-u_n\|_\sV^2 - 2\hat{c}_1 \|u-u_n\|_\sV - \hat{c}_2 \leq 0.$$
Solving the quadratic inequality then yields (\ref{eq:pron-errbnd-pr}).  \qquad
\end{proof}

We now make some observations about the {\it a posteriori} error bounds derived above.  First, the computational cost to compute $\hat{\Delta}_u^{\pr}$ and $\hat{\Delta}_\lambda^{\pr}$ as in \cite{HSW2012} relies greatly on the particular choice of the projection $\Pi$ and the scalar product $(\cdot,\cdot)_\sQ$ on $\sQ$.  Second, we note that $r_{\rm i}(\lambda)$ is close to zero presuming that $(u_n^\pr,\lambda_n^\pr)$ is a good approximation to $(u,\lambda)$.  However, the subsequent steps  in the derivation of the error bounds --- in particular the use of the Cauchy-Schwarz inequality in (\ref{eq:cauchy-schwarz}) and the application of (\ref{eq:dual-err}) --- cause the resulting error estimators to lose sharpness.  The culprit, as we shall see in the next section, lies in the fact that $K_n$ is not necessarily in $\sK$, i.e., $u_n$ satisfies (\ref{eq:gkkt-app-2}) but not (\ref{eq:gkkt2}).  We shall return to these remarks in {{subsequent sections}}.  In the meantime, however, these considerations behoove us to develop an alternative approach that  preserves the near-orthogonality of the terms in (\ref{eq:almostorth}) and concurrently reduces the computational cost.

\section{Approximation and Error Estimation: A Duality Approach}
\label{sec:duality}
We now introduce a {\it dual} or {\it auxiliary} problem which provides strictly feasible approximations to our original primal problem.  We reiterate that $B$ is assumed to be bijective. Our point of departure is the algebraic formulation of the standard variational inequality {\bf A1}: 
\begin{problem}[\underline{\bf A1}]
{\it Given $\usK := \left\{ \uv \in \RcNV \; | \; {\uB\,\uv \leq {\ug} }\right\}\!,$ find $\uu \in \usK$ such that 
\begin{equation}
\label{eq:vi-alg}
(\uv-\uu)^T \uA\, \uu  \geq (\uv-\uu)^T \uf, \quad\forall\;\uv\in \usK.
\end{equation}
If $\uA$ is symmetric positive-definite, then (\ref{eq:vi-alg}) is the optimality condition of the corresponding minimization problem
\begin{equation}
\label{eq:min-alg}
\uu = \arg \min_{\uv \in \usK} {\textstyle \frac12} \uv{}^T\uA\,\uv - \uv{}^T\uf.
\end{equation}
}
\end{problem}

\noindent Here, all discrete quantities are defined as in Sec.~\ref{sec:alg_prob}.  

\subsection{The Dual Problem}

We now introduce a {\it dual} or {\it slack} variable $s \in \sM' {\subset Q'}$ given by
\begin{equation}
\ang{s,q} = \ang{g-Bu,q}, \quad \forall \ q \in \sQ.
\end{equation}
Since $B$ is bijective, we have
$ u = B^{-1} \left(g-s\right){\in \sK}.$
Defining the corresponding finite element vector $\us\in \RcNQ_+$ as 
\begin{equation}
\label{eq:slack}
\us := \ug - \uB\, \uu ,
\end{equation}
we can then define our {\it dual\,} problem as:
\begin{problem}[\uPdu]
{\it Find $\us \in \RcNQ_+$ such that 
\begin{equation}
\label{eq:vi-du}
(\uxi-\us)^T \utA\, \us  \geq (\uxi-\us)^T \utf, \quad\forall\;\uxi\in  \RcNQ_+.
\end{equation}
where $\utA:\RcNQ \to \RcNQ$ and $\utf \in \RcNQ$ are given by
\label{eq:tilde}
\begin{eqnarray} 
\utA :=  \uB^{-T}\uA\, \uB^{-1},\qquad
\utf  :=  \uB^{-T}\uA\, \uB^{-1} \ug - \uB^{-T} \uf.
\end{eqnarray}
If $\uA$ is symmetric positive-definite, then (\ref{eq:vi-du}) is the optimality condition of the corresponding minimization problem
\begin{equation}
\label{eq:min-alg-du}
\us = \arg \min_{\uxi \in \RcNQ_+} {\textstyle \frac12} \uxi{}^T\utA\,\uxi - \uxi{}^T\utf.
\end{equation}
}
\end{problem}

We now show that  \uPpr\ and \uPdu\ are equivalent in:

\begin{corollary} If $\us$ and $\uu$ are related by (\ref{eq:slack}), then $\us$ is the solution to \uPdu\ if and only if $\uu$ is the solution to \uPpr.
\end{corollary}

\begin{proof}
We first show that (\ref{eq:vi-alg}) implies (\ref{eq:vi-du}).  Since $B$ is bijective, for any $\uv \in \usK$ there exists a unique $\uxi \in \RcNQ_+$ such that
\begin{equation}
\label{eq:vandxi}
\uxi = \ug - \uB \uv.
\end{equation}
Noting from (\ref{eq:slack}) and (\ref{eq:vandxi}) that $\uu = \uB^{-1} (\ug-\us)$ and $\uv = \uB^{-1} (\ug-\uxi)$, we have
$ \uv - \uu = \uB^{-1}(\us-\uxi).$
Substituting this into (4.1) then yields
$$ (\us - \uxi)^T \uB^{-T} \uA\, \uB^{-1} (\ug - \us) \geq (\us - \uxi)^T \uB^{-T} \uf.$$ 
The desired result (\ref{eq:vi-du}) then directly follows from (\ref{eq:tilde}).  By a similar technique we can also readily show that (\ref{eq:vi-du}) implies (\ref{eq:vi-alg}).  We simply substitute (\ref{eq:slack}),  (\ref{eq:tilde}), and (\ref{eq:vandxi}) into (\ref{eq:vi-du}).  Rearrangement of terms then yields (\ref{eq:vi-alg}).  This completes the proof.  
\end{proof}

\subsection{Dual Approximation}
\label{sec:dual-approximation}

Let the approximation space $\usS_n \subset \RcNQ_+ $ be given in terms of basis functions $\uzeta_i$, $1 \leq i \leq n_\sS,$ by 
\begin{equation}
\usS_n := {\rm span}_+ \{ \uzeta_i \in \RcNQ_+, \ 1 \leq i \leq n_\sS\} 
= \{ \usigma \in \RcNQ_+ \ | \ \usigma = \textstyle\sum\limits_{i=1}^{n_\sS} \underline{\uc}_i \uzeta_i \ {\rm and} \ \underline{\uc} \in \mathbb{R}^{n_\sS}_+\}.
\end{equation}  
Here, we assume that the basis functions are linearly independent (though not necessarily orthogonal).  We now consider the following approximation to Problem \uPdu: 

\begin{problem}[\uPdun]
{\it Find $\us_n \in \usS_n$ such that
\begin{equation}
\label{eq:vi-du-app}
(\uxi{}_n-\us{}_n)^T \utA\, \us{}_n  \geq (\uxi{}_n-\us_n)^T \utf{}_n, \quad\forall\;\uxi{}_n\in \usS{}_n.
\end{equation}
If $\uA$ is symmetric positive-definite, then (\ref{eq:vi-du-app}) is the optimality condition of the corresponding minimization problem
\begin{equation}
\label{eq:min-alg-app-du}
\us_n = \arg \min_{\uxi_n \in \usS{}_n} {\textstyle \frac12} \uxi_n{}^T\utA\,\uxi_n - \uxi_n{}^T\utf.
\end{equation}
}

\end{problem}

\noindent 
Following the notation in Sec.~\ref{sec:alg_prob}, we note that
$ \us{}_n = \textstyle\sum_{i=1}^{n_\sS} \uus{}_{n,i} \uzeta{}_i$,
where $\uus{}_n \in\RcnQ_+$ is the solution of 
\begin{problem}[\uuPdun]
{\it Find $(\uus_n,\uulambda_n^\du) \in \mathbb{R}^{n_\sS} \times \mathbb{R}^{n_\sS}$ such that
\begin{subequations}
\label{eq:rkkt-du}
\begin{alignat}{7}
\label{eq:rkkt-du-1}
\uutA{}_n \,\uus{}_n + \uulambda_n^\du & \ = \ & & \ \uutf{}_n \\
\label{eq:rkkt-du-2}
\uus_n & \ \geq \ & & \ 0 \\
\label{eq:rkkt-du-3}
\uulambda{}_n^\du & \ \geq \ &&  \ 0 \\
\uus_n^T {\uulambda{}_n^\du} &  \ = \ & & \ 0,
\label{eq:rkkt-du-4}
\end{alignat}
\end{subequations}
where ${\uutA}{}_{n,ij} = \uzeta{}_i^T \utA\, \uzeta_j$,  and ${\uutf}{}_i = \uzeta{}^T_i{\utf}$, for $i,j = 1,\dots,n_{n_\sS}$.  }
\end{problem}

We now briefly remark on the properties of our dual problem and its approximation.  First, we note  from our assumptions on $A$, $B$, and $\uzeta_i$, that $\utA$ and $\utA_n$ are positive-definite and bounded.  Therefore, the existence, uniqueness, and boundedness of the solutions to \uPdu\ and \uPdun\  follows directly from Thm.~\ref{thm:LS67}. 

Second, we remark on a fundamental difference  between our primal problem \uuPprn\
(defined in Sec.~\ref{sec:primalapprox}) and our dual problem \uuPdun\ defined above.  We note that \uuPprn\ is obtained through an {\it optimize-then-discretize} approach: we optimize {\bf A1} to obtain {\bf A2} and {\Ppr}, and then discretize to obtain \uPpr.  We subsequently introduce an approximation  to \uPpr, thus obtaining an approximation $\lambda^\pr_n$ to the KKT multiplier $\lambda$.  

On the other hand, \uuPdun\ is derived using a {\it discretize-then-optimize} approach: we discretize {\bf A1} to obtain {\bf \underline{A1}}, re-write the problem in terms of the slack variable to obtain \Pdu, then optimize to obtain \uPdu.  An ``algebraic"  approximation is subsequently applied to \uPdu\ to obtain \uPdun.  Note that the KKT multiplier $\uulambda{}^\du_n$ serves only to enforce the constraint $\uus{}_n \geq 0$.   In other words,  $\uulambda{}^\du_n$ are {\it not} coefficients of  corresponding basis functions for the exact KKT multiplier $\lambda$, and \uPdun does {\it not} provide a direct approximation to $\lambda$.  
 
What then is the purpose of the dual problem?  Given $\us_n$, we let $s_n \in \sM'$ be given by $s_n =  \textstyle\sum_{i=1}^{\cN_\sV} \us_{n,i} (g-B\Phi_i),$ and  further define a new approximation, $u^\du_n$ to $u$ given by 
$u_n^\du = \textstyle\sum_{i=1}^{\cN_\sV} \uu^{\du}_{n,i} \Phi_i,$
and 
\begin{equation}
\uu^{\du}_n := \uB^{-1} \left( \ug-\us_n\right).
\label{eq:udun}
\end{equation} 
The purpose of this dual approximation can be better understood by considering\begin{corollary}
\label{thm:dual-app-u-sigma}  
For any $\sigma \in \sM'$, the function $u^\sigma \in \sV$ given by
\begin{equation}
\label{eq:stou}
u^\sigma := B^{-1}(g-\sigma) 
\end{equation}
is in $\sK$.  That is, it satisfies 
\begin{equation}
\label{eq:sineq}
\ang{g - Bu^\sigma,q}  \geq  0, \quad \forall \ q \in \sM.
\end{equation}
\end{corollary}
\begin{proof}
From (\ref{eq:stou}), we have $\sigma = g-Bu^\sigma$.  Eqn.~(\ref{eq:sineq}) therefore directly follows from the definition of $\sM'$ in (\ref{eq:cone}). \qquad
\end{proof}

\noindent We recall that the primal approximation $\uu^\pr_n \in \usK_n$  does {\it not} necessarily satisfy (\ref{eq:sineq}) (i.e., in general $\usK_n \not\subset \usK$).  This led to difficulties in Sec.~\ref{sec:primal} in the derivation of {\it a posteriori} error estimates, causing a loss of sharpness as well as necessitating the introduction of the nonlinear projection $\Pi$ and the use of the Cauchy-Schwarz inequality in (\ref{eq:cauchy-schwarz}).   On the other hand, the dual approximation (\ref{eq:udun}) provides strictly feasible approximations  $\uu^\du_n$ to $\uu$ by virtue of Corollary~\ref{thm:dual-app-u-sigma}.   As we shall see in the next section, this fact greatly simplifies the development of a posteriori error bounds.  

\subsection{Error Estimation via Duality}
We begin by defining the residual 
\begin{equation}
r(v) := \ang{f,v} - \ang{Au^\du_n,v} - \ang{B v, \lambda^\pr_n}, \qquad v \in \sV.
\end{equation}
where $u_n^\du$ is defined as in the previous section and $\lambda_n^\pr$ is the solution to \Pprn.  We then derive error bounds for our primal-dual approximation $(u^\du_n,\lambda^\pr_n)$ in  
\begin{proposition}  Let
\label{prop:dual-error-bound}
\begin{equation}
\label{eq:duerr}
\tilde{d}_1 :=  \frac{\| r \|_{\sV'}}{2\alpha} \qquad
\tilde{d}_2 := \frac{\ang{s_n,\lambda_n^\pr}}{\alpha} .
\end{equation}
The errors can then be bounded by 
\begin{subequations}
\begin{alignat}{8}
\label{eq:errbnd-u-pr-du}
\|u-u_n^\du\|_\sV & \leq & \ & \tilde{\Delta}_u^{\pr,\du} & \ & := &  \ & \tilde{d}_1 + \sqrt{\tilde{d}_2^2 + \tilde{d}_2} \\
\label{eq:errbnd-l-pr-du}
\|\lambda-\lambda_n^\pr\|_\sQ & \leq & \ & \tilde{\Delta}_\lambda^{\pr,\du} & \ & := &  \ & \frac{1}{\beta} \left( \| r \|_{\sV'} + \gamma \tilde{\Delta}_u^{\pr,\du} \right) 
\end{alignat}
\end{subequations}
\end{proposition}
\begin{proof}
We shall again omit the superscripts ``pr" and ``du" in this proof.  We thus emphasize that here, $(u_n,\lambda_n)$ refers to our {\it primal-dual} approximation $(u_n^\du,\lambda_n^\pr)$.  From (\ref{eq:rkkt-du}), (\ref{eq:gkkt1}), (\ref{eq:alpha-gamma}) we then have 
\begin{equation}
\alpha \| u-u_n\|^2_{\sV} \leq  r(u-u_n) - \ang{B(u-u_n),\lambda-\lambda_n} .
\end{equation}
We then note from (\ref{eq:gkkt4}) and (\ref{eq:sineq}) that 
$\ang{B(u-u_n),\lambda} = \ang{g-Bu_n,\lambda} \geq 0,$
and that 
$\ang{B(u-u_n),\lambda_n} \leq \ang{g,\lambda_n} - \ang{g-s_n,\lambda_n} = \ang{s_n,\lambda_n}.$
It thus follows that 
\begin{equation}
\label{eq:sbound}
\alpha \| u-u_n\|^2_{\sV} \leq  \|r\|_{\sV'}\|u-u_n\|_{\sV}+\ang{s_n,\lambda_n} .
\end{equation}
Using (\ref{eq:duerr})  and  solving the quadratic inequality (\ref{eq:sbound}), we obtain (\ref{eq:errbnd-u-pr-du}).  The remaining result (\ref{eq:errbnd-l-pr-du}) follows directly from Lemma \ref{lem:dual-err} and (\ref{eq:errbnd-u-pr-du}).  \qquad 
\end{proof}

\section{Application to the Reduced Basis Method}
\label{sec:rb}

As indicated in the introduction, we now apply the techniques in Sections \ref{sec:primal} and  \ref{sec:duality} to the reduced basis method.  The RB method is a model order reduction technique intended for use in real-time optimisation, control, or characterisation of systems governed by parametrized partial differential equations.    
The RB method constructs inexpensive yet certified surrogates for the exact (i.e., ``truth") solution by focusing on the solution manifold induced by the parametrized PDE. Rigorous a posteriori error bounds are then derived based on relaxations of the error-residual equation.  

Typically, both RB approximations and error bounds are computed using an offline-online strategy enabling highly efficient (i.e., at marginal online cost) computations of the approximations and error bounds. Also, RB approximations and error bounds are, in practice, intimately linked through a greedy approach, in which the (online-) inexpensive error bounds are used to construct the subsequent approximation spaces systematically and (quasi-)optimally (see, e.g., \cite{BMP+2012,BCD+2011}).  In order to facilitate the comparison of the proposed and existing approaches, however, we shall neither be discussing nor applying the greedy approach in this work.  

In Sec.~\ref{sec:prob-state}, we state the problem and present the required assumptions on the nature of the parametric dependence of the PDE.  In the subsequent sections, we summarise the key elements for the RB approximation of the primal problem  in Sec.~\ref{sec:primal-approx} (\cite{HSW2012}), the dual problem in Sec.~\ref{sec:dual-approx}, and for the RB error estimation using the primal-dual approach in Sec.~\ref{sec:a-post-error}.  Here, our focus is on the choice of the approximation spaces, on the required bounds to the coercivity and continuity constants, and on the offline-online computational procedure.  

\subsection{Problem Statement}
\label{sec:prob-state}

Let $\cD \subset \mathbb{R}^{p}$ be a prescribed $p$-dimensional, compact parameter set. We introduce a parameter $\mu \in \cD \subset \mathbb{R}^{p}$, and assume that $A$, $f$, and $g$ depend affinely on $\mu$:  
\begin{equation}
\label{eq:affine}
A(\mu) = \textstyle\sum\limits_{k=1}^{Q_a} \Theta_a^k (\mu) A^k, \qquad
f(\mu)= \textstyle\sum\limits_{k=1}^{Q_f} \Theta_f^k (\mu) f^k, \qquad
g(\mu)= \textstyle\sum\limits_{k=1}^{Q_g} \Theta_g^k (\mu) g^k ,
\end{equation}
where $Q_a, Q_f, Q_g \in \mathbb{N}$ are assumed to be small, and  
the parameter-dependent coefficient functions $\Theta_a^k(\mu)$, $\Theta_f^k(\mu)$, and $\Theta_g^k(\mu)$ are continuous over the parameter set $\cD$.  We also assume that the mappings $A^k:\sV \to \sV'$, $f^k:\sV \to \mathcal{R}$, and $g^k:\sQ \to \mathcal{R}$ are parameter-independent, linear, and continuous.

Furthermore, we assume that for all $\mu \in \cD$, $A(\mu)$ satisfies (\ref{eq:alpha-gamma}) with continuity constant $\gamma(\mu)$ and coercivity constant $\alpha(\mu)$, and also that $f(\mu) \in \sV'$ and $g(\mu) \in \sQ'$.  Finally, as indicated in the introduction, we assume that $B$ is parameter-{\it independent}.  

We  consider the parametrized forms of our primal problem \Ppr \ (see \cite{HSW2012})
\begin{problem}[\Ppr$(\mu)$]
\label{pr:mixed-mu}
Find $(u(\mu),\lambda(\mu)) \in \sV \times \sM$ such that
\begin{subequations}
\label{eq:spp-mu}
\begin{alignat} {6}
\label{eq:spp1-mu}
\ang{A(\mu)u(\mu),v}  + \ang{Bv, \lambda(\mu)} & = & \ &  \ang{f(\mu),v},  & & \quad\forall\;v\in \sV \\
\label{eq:spp2-mu}
\ang{Bu(\mu),q-\lambda(\mu)} & \leq & \ & \ang{g(\mu),q-\lambda(\mu)}, & & \quad\forall\;q\in \sM;
\end{alignat}
\end{subequations}
\end{problem}
and of our dual problem  \uPdu
\begin{problem}[\bf \uPdu$(\mu)$]
{\it Find $\us(\mu) \in \RcNQ_+$ such that 
\begin{equation}
\label{eq:vi-du-mu}
(\uxi-\us(\mu))^T \utA(\mu) \us(\mu)  \geq (\uxi-\us(\mu))^T \utf(\mu), \quad\forall\;\uxi\in  \RcNQ_+,
\end{equation}
where $\utA(\mu):\RcNQ \to \RcNQ$ and $\utf \in \RcNQ$ are given by
$$
\utA(\mu) := \uB^{-T}\uA(\mu) \uB^{-1} \qquad
\utf(\mu) := \uB^{-T}\uA(\mu) \uB^{-1} \ug(\mu) - \uB^{-T} \uf(\mu).
$$}
\end{problem}
We now consider primal-only  and primal-dual reduced basis approximations to (\ref{eq:spp-mu}) and (\ref{eq:vi-du-mu}) based on the methods of Sec.~\ref{sec:primal} and \ref{sec:duality}, respectively.  

\subsection{Primal Approximation and A Posteriori Error Estimation}
\label{sec:primal-approx}

As discussed in Sec.~\ref{sec:primal}, the (primal) reduced basis spaces $\sV_n \subset \sV$ and $\sQ_n \subset \sQ$ must be chosen such that the associated inf-sup constant $\beta_n > 0$.  
We first choose $\sQ_n$ to be the space spanned by ``snapshots" of the multiplier $\lambda$ for $n$ values of the parameter:
\begin{equation}
\label{eq:rb-pr-lambda}
\sQ_n = {\rm span} \{ \psi_i \in \sM, \ 1 \leq i \leq n_\sQ \} = {\rm span} \{ \lambda(\mu_i), ~i=1,\ldots,n\}. 
\end{equation}  
Here, the $\psi_i$ are chosen to be linearly independent but not necessarily orthogonal, and therefore (in general) $n_\sQ \leq n.$  We next choose $\sV_n$ to be the space spanned by snapshots of the field variable $u$ {\it and} by additional ``supremizing" functions $t_k$: 
\begin{eqnarray}
\label{eq:rb-pr-u}
\sV_n & = & {\rm span} \{ \varphi_j \in \sV, \ 1 \leq j \leq n_\sV\} \\ 
& = & {\rm span} \{ u(\mu_j), t_k, ~j=1,\ldots,n,~k=1,\ldots,n_{\sup{}} \},
\end{eqnarray}
Here, the $\varphi_j$ are assumed to be mutually orthogonal, that is, they are computed using a Gram-Schmidt orthogonalization procedure.  Furthermore, $n_{\sup}$ is the number of additional ``supremizing functions" required to ensure inf-sup stability.  For more details on the choice of the supremizing functions, we refer the reader to \cite{RV2007} and \cite{GV2011a}; in the special case of the model problems to be discussed in Sec.~\ref{sec:example}, we also refer to \cite{HSW2012}.

We further define the reduced convex cone \cite{HSW2012}
\begin{equation}
\label{eq:rb-pr-cone}
\sM_n  := {\rm span}_+ \{ \psi_i,~i=1,\ldots,Nn\}.
\end{equation}
Our reduced basis approximation $(u_n(\mu),\lambda_n(\mu))\in \sV_n \times \sQ_n$ to $(u(\mu),\lambda(\mu)$) is then given by the parametrized form of {\bf P}$_n^\pr$ \cite{HSW2012}:
\begin{subequations}
\label{eq:pprn-mu}
\begin{alignat}{7} 
\label{eq:spp-rb}
\ang{ A(\mu)u_n(\mu),v_n } + \ang{Bv_n, \lambda_n(\mu)}  \ & = & & \ \ \ang{f(\mu),v_n}, \quad && \forall\;v_n\in \sV_n, \\
\ang{B u_n(\mu),q_n-\lambda_n(\mu)} \ & \leq & & \ \ \ang{g(\mu),q_n-\lambda_n(\mu)}, \qquad && \forall\;q_n\in \sM_n.
\end{alignat}
\end{subequations}
Equivalently, the coefficients (with respect to the basis functions $\varphi_i, \psi_j$) can be obtained by solving the parametrized form of \uuPprn: Find $(\underline{\underline{u}}{}_n(\mu),\underline{\underline{\lambda}}{}_n(\mu)) \in \RcnV\times \RcnQ$ such that
\begin{subequations}
\label{eq:kkt-pr-rb}
\begin{eqnarray} 
{ \uuA{}_n(\mu){\uuu}{}_n(\mu)} + {\uuB{}_n^T \uulambda{}_n(\mu)} & \ = \ &  {\uuf{}_n(\mu)} \\
\uug{}_n(\mu) - \uuB{}_n \uuu{}_n(\mu)& \geq & 0 \\
\uulambda{}_n(\mu)& \geq & 0 \\
\uulambda{}^T_n(\mu) \big( \uug{}_n(\mu) - \uuB{}_n \uuu{}_n(\mu)\big) & = & 0,
\end{eqnarray}
\end{subequations}
where, for $i,j = 1,\dots,{n_\sV}$ and $k=1,\dots,{n_\sQ}$, 
\begin{subequations}
\label{eq:kkt-pr-rb-def}
\begin{alignat}{12}
{\uuA}{}_{n,ij}(\mu) & = & \ & \uvarphi{}_i^T \uA(\mu) \, \uvarphi_j, & \qquad & {\uuB}{}_{n,jk} & = & \ & \upsi{}_j^T \uB\, \uvarphi_k, \\
{\uuf}{}_i(\mu)  & = & \ & \uvarphi{}^T_i{\uf}(\mu), & \qquad &
{\uug}{}_k(\mu)  & = & \ & \upsi{}^T_k{\ug}(\mu).
 \end{alignat}
\end{subequations}

With our assumptions on $\sQ_n$ and $\sV_n$ (i.e., on inf-sup stability and on the linear independence of the corresponding basis functions), it follows that Corollary~\ref{cor:bhr} holds and a unique solution exists.  

We now apply the approach presented in \cite{HSW2012} and summarised in Sec.~\ref{sec:primal} to compute {\it a posteriori} error bounds for the (primal-only) reduced basis approximation $(u^\pr_n(\mu),\lambda^\pr_n(\mu))$.

We assume that for all $\mu \in \cD$, we have computationally inexpensive lower (respectively, upper) bounds to the truth coercivity (resp., continuity) constant:  
\begin{equation}
\label{eq:lb-ub}
\alpha_{\rm LB}(\mu) \leq \alpha(\mu), \qquad \gamma_{\rm UB}(\mu) \geq \gamma(\mu).
\end{equation}
Note that the inf-sup constant $\beta$ does not depend on the parameter  since $B$ is assumed to be $\mu$-independent.  Applying Prop.~\ref{prop:HSW}, we then have
\begin{corollary}   For $\mu \in \cD$, let
\begin{equation*}
\label{eq:res-est-mu}
\delta_0(\mu) := \| r_{\rm e}(\mu)\|_{\sV'}, \qquad \delta_1:= \| \Pi({e}_{\rm i}(\mu)) \|_\sQ, \qquad \delta_2(\mu) := \big({\lambda_n(\mu),\Pi({e}_{\rm i}(\mu))\big)_\sQ}
\end{equation*}
and 
\begin{equation*}
\label{eq:pr-error-constant}
{c}_1(\mu) := \frac{1}{2\alpha_{\rm LB}(\mu)} \left( \delta_0 + \frac{\gamma_{\rm UB}(\mu) \delta_1  }{\beta}\right),\quad
{c}_2(\mu) :=  \frac{1}{\alpha_{\rm LB}(\mu)} \left( \frac{\delta_0(\mu)\delta_1(\mu)}{\beta}   + \delta_2(\mu)\right).
\end{equation*}
The errors in $(u_n^\pr(\mu),\lambda_n^\pr(\mu))$ with respect to $(u(\mu),\lambda(\mu))$ can then be bounded by 
\begin{subequations}
\begin{alignat}{8}
\label{eq:pr-bound-u}
\|u(\mu)-u_n^\pr(\mu)\|_\sV & \leq & \ & {\Delta}_u^{\pr}(\mu) & \ & := &  \ & {c}_1(\mu) + \sqrt{{c}_1^2(\mu) + {c}_2}(\mu), \\
\label{eq:pr-bound-l}
\|\lambda(\mu)-\lambda_n^\pr(\mu)\|_\sQ & \leq & \ & {\Delta}_\lambda^{\pr}(\mu) & \ & := &  \ & \frac{1}{\beta} \left( \delta_0(\mu) + \gamma_{\rm UB}(\mu) {\Delta}_u^{\pr}(\mu) \right).
\end{alignat}
\end{subequations}
\end{corollary}
As mentioned in Sections~\ref{sec:primal} and \ref{sec:duality}, the ``exact" inequality constraint (\ref{eq:gkkt2}) is in general not satisfied by our primal-only approximation, thus leading to difficulties in error estimation.  We thus pursue the primal-dual approach of Sec.~\ref{sec:duality} to obtain strictly feasible approximations to $u(\mu)$ and associated simpler {\it a posteriori} error bounds.

\subsection{Dual Approximation}
\label{sec:dual-approx}

We define our dual reduced basis space $\usS_n \subset \RcNQ_+$ to be the span of snapshots of the slack variable $s$:
\begin{equation}
\usS_n = {\rm span}_+ \{ \uzeta_i \in \RcNQ_+, \ 1 \leq i \leq n_\sS \} = {\rm span}_+ \{ \us(\mu_i) \in \RcNQ_+, ~i=1,\ldots,n_\sS'\}. 
\end{equation}
Our reduced basis approximation $\us_n(\mu)$ to $\us(\mu)$ is then given by: 
\begin{problem}[\uPdun$(\mu)$]
{\it Find $\us_n(\mu) \in \usS_n$ such that 
\begin{equation}
\label{eq:vi-du-mu-n}
(\uxi-\us_n(\mu))^T \utA(\mu) \us_n(\mu)  \geq (\uxi-\us_n(\mu))^T \utf(\mu), \quad\forall\;\uxi\in  \usS_n.
\end{equation}}
\end{problem}
Equivalently, the coefficients (with respect to the basis functions $\uzeta_i$) can be obtained by solving the parametrized form of \uuPdun: 
Find $(\uus_n(\mu),\uulambda_n^\du(\mu)) \in \mathbb{R}^{n_\sS} \times \mathbb{R}^{n_\sS}$ such that
\begin{subequations}
\label{eq:kkt-du-rb-sys}
\begin{alignat}{7}
\uutA{}_n(\mu) \,\uus{}_n(\mu) + \uulambda{}_n^\du(\mu) & \ = \ & & \ \uutf{}_n(\mu) \\
\uus{}_n(\mu) & \ \geq \ & & \ 0 \\
\uulambda{}_n^\du(\mu) & \ \geq \ &&  \ 0 \\
\uus{}_n^T(\mu) {\uulambda{}_n^\du(\mu)} &  \ = \ & & \ 0,
\end{alignat}
\end{subequations}
where, for $i,j = 1,\dots,n_{\sS}$, 
\begin{equation} 
\label{eq:tilde-reduced}
{\uutA}{}_{n,ij} =  \uzeta{}_i^T \utA\, \uzeta_j, \qquad \qquad
{\uutf}{}_i = \uzeta{}^T_i{\utf}.
\end{equation}  
Under the assumption that the basis functions $\uzeta_i$ are linearly independent, Theorem~\ref{thm:LS67} holds and there exists a unique, bounded solution to \uPdun$(\mu)$.  

As before, given the approximation $s_n(\mu)$, we then compute the corresponding approximation $u^\du(\mu)$ to $u(\mu)$ by
\begin{equation}
u^{\du}(\mu) = B^{-1}\big(g(\mu)-s_n(\mu)\big).
\end{equation}
In contrast to $u^\pr(\mu)$, the dual approximation satisfies $u^\du(\mu) \in \sK$  for all $\mu \in \cD$.

\subsection{Primal-Dual A Posteriori Error Estimation}
\label{sec:a-post-error}

We now apply the approaches presented in Sec.~\ref{sec:duality} to compute computationally inexpensive {\it a posteriori} error bounds for our primal-dual reduced basis approximation $(u_n^\du(\mu),\lambda_n^\pr(\mu))$.  Using (\ref{eq:lb-ub}) and applying Prop.~\ref{prop:dual-error-bound}, we then have
\begin{corollary}  For $\mu \in \cD$, let
\begin{subequations}
\label{eq:du-error-constant}
\begin{eqnarray}
d_1(\mu) := \frac{\| r(\mu) \|_{\sV'}}{2\alpha_{\rm LB}(\mu)}, \qquad d_2(\mu) :=  \frac{\ang{s_n(\mu),\lambda_n^\pr(\mu)}}{\alpha_{\rm LB}(\mu)} .
\end{eqnarray}
\end{subequations}
The errors in $(u_n^\du(\mu),\lambda_n^\pr(\mu))$ with respect to $(u(\mu),\lambda(\mu))$ can then be bounded by 
\begin{subequations}
\begin{alignat}{8}
\label{eq:du-bound-u}
\|u(\mu)-u_n^\du(\mu)\|_\sV & \leq & \ & {\Delta}_u^{\pr,\du}(\mu) & \ & := &  \ & d_1(\mu) + \sqrt{d_2(\mu)^2 + d_2(\mu)}, \\
\label{eq:du-bound-l}
\|\lambda(\mu)-\lambda_n^\pr(\mu)\|_\sQ & \leq & \ & {\Delta}_\lambda^{\pr,\du}(\mu) & \ & := &  \ & \frac{1}{\beta} \left( \| r(\mu) \|_{\sV'} + \gamma_{\rm LB}(\mu) {\Delta}_u^{\pr,\du}(\mu) \right).
\end{alignat}
\end{subequations}
\end{corollary}
Now that we have developed our reduced basis approximations and corresponding {\it a posteriori\/} error bounds, we turn to the issue of computational efficiency.  

\subsection{Offline-Online Computational Procedure}

In the reduced basis approach, the offline-online computational strategy relies on the affine $\mu$-dependence of the quantities involved.  With this assumption, all $\mu$-independent quantities (for example, in (\ref{eq:affine}) can be formed and stored within a computationally expensive {\it offline} phase.  This stage, the cost of which depends on the large finite element dimension $\cN$, is performed only once.  For any given parameter $\mu \in \cD$, the RB approximation is then computed in a highly efficient online phase.  Ideally, the computational cost of the online phase would not depend on $\cN$ but only on the considerably smaller dimension of the RB approximation space.  

Much of this machinery is by now standard in RB methods (see, for example, \cite{RHP2008}).  However, we note that the techniques we present here are non-standard; we must thus elaborate on the offline-online computational decomposition for our primal-dual approach in greater detail.  For more details on the primal-only approach, we refer the reader to \cite{HSW2012}.

For clarity, we discuss the approximation and {\it a posteriori} error estimation stages separately.  We begin with the {\it approximation} stage.  From (\ref{eq:tilde}), (\ref{eq:affine}), (\ref{eq:kkt-pr-rb-def}), and (\ref{eq:tilde-reduced}), we note that
\begin{subequations}
\label{eq:affine-primal-dual}
\begin{alignat}{12}
\uuA{}_{n,ij}(\mu) & = &\ & \textstyle\sum\limits_{q=1}^{Q_a} \Theta_a^q(\mu) \ang{A^q\varphi_j,\varphi_i}, & \qquad & & 
\uutA{}_{n,lm}(\mu) & = & \ &\textstyle\sum\limits_{q=1}^{Q_a} \tTheta_a^q(\mu) \ang{\tA^q\zeta_m,\zeta_l}, \\
\uuf{}_{n,i}(\mu)& = &\ & \textstyle\sum\limits_{q=1}^{Q_f} \Theta_f^q (\mu) \ang{f^q,\varphi_i}, & \qquad & & 
\uutf{}_{n,l}(\mu)& = &\ & \textstyle\sum\limits_{q=1}^{\tQ_f} \tTheta_f^q (\mu) \ang{\tf^q,\zeta_l}, \\
\uug{}_{n,k}(\mu)& = &\ & \textstyle\sum\limits_{q=1}^{Q_g} \Theta_g^q (\mu) \ang{g^q,\psi_k} , 
\end{alignat}
\end{subequations}
where $\tQ_f = Q_f + Q_aQ_g$,  $\tA^q := B^{-T}A^q B^{-1},$ and 
\begin{equation*}
\tf^q := 
\begin{cases}
-B^{-T}f^q, & q = 1,\dots,Q_f \\
B^{-T}A^{q'}B^{-1} g^{q''},  & 1 \leq q' \leq Q_a, \ 1 \leq q'' \leq Q_g, \  q = Q_f + (q'-1)Q_g + q''.
\end{cases}
\end{equation*}
Thus, in the {\it offline} stage, we solve \uPpr and \uPdu (i.e., (\ref{eq:kkt}) and (\ref{eq:vi-du})) for the snapshots and compute the basis functions  $\uvarphi{}_i,\ i=1,\dots,n_\sV$, $\upsi{}_k,\ k=1,\dots,n_\sQ$, and $\uzeta{}_l, \ l=1,\dots,n_\sS$.  We then compute and store the $\mu$-independent quantities
\begin{subequations}
\label{eq:assembling}
\begin{alignat}{12}
\uuA{}^q_{n,ij} \ & = \ \ & & \uvarphi{}_i^T \uA^q\uvarphi{}_j, &\quad && \uutA{}^q_{n,lm} \ \ & = \ \ & & \uzeta{}_l^T \utA{}^q\uzeta{}_m, & \quad & q = 1,\dots,Q_a, \\[0.8ex]
\uuB{}_{n,kj} \ \ & = \ \ & & \upsi{}_k^T \uB\uvarphi{}_j, &\quad && \uug{}^q_{n,k} \ \ & = \ \ & & \upsi{}_k^T \ug^q, &\quad & q = 1,\dots,Q_g, \\
\uuf{}^q_{n,i} \ \ & = \ \  & & \uvarphi{}_i^T \uf^q, &\quad && \uutf{}^{q'}_{n,l} \ \ & = \ \  & & \uzeta{}_l^T \utf{}^{q'}, &\quad & q = 1,\dots,Q_f, \ q' = 1,\dots,\tQ_f.
\end{alignat}
\end{subequations}
for $i,j = 1,\dots,n_\sV$, $k=1,\dots,n_\sQ$, and $l,m=1,\dots,n_\sS$.  The dominant computational cost to compute (and store) the required quantities in (\ref{eq:assembling}) is then  $\cO(Q_a n_\sV^2 \mathcal{N}_\sV^{*}+ n_\sV n_\sQ \mathcal{N}_\sV^{*}\mathcal{N}_\sQ^{*}+Q_a n_\sS^2\mathcal{N}_\sQ^{*})$ (and $\cO(Q_a n_\sV^2+n_\sV n_\sQ + Q_a n_\sS^2)$).  

In the {\it online} stage, we then compute the summations in (\ref{eq:affine-primal-dual}) (at cost $\cO(Q_a n_\sV^2+Q_a n_\sS^2)$), and solve (\ref{eq:kkt-pr-rb}) as well as (\ref{eq:kkt-du-rb-sys}) at a cost that depends only on $n_\sV, n_\sQ$, $n_\sS$, and on the complexity of the parameter dependence (through $Q_a, Q_f, Q_g$), and is independent of the dimension of the finite element problem.   

We now turn to the {\it a posteriori} error estimation stage.  The required lower bound to the coercivity constant, $\alpha_{\rm LB}(\mu),$ in (\ref{eq:du-error-constant}) can be calculated using the (now) standard successive constraints method (SCM) proposed in \cite{HRS+2007} and further improved in \cite{HKC+2010}.  The offline-online calculation of the dual norm of the residual in $d_1$ of (\ref{eq:du-error-constant}) is an application of now standard RB techniques that can be found in, e.g., \cite{PR2007a}, \cite{RHP2008}.  Turning now to $d_2$ in (\ref{eq:du-error-constant}), we note that 
\begin{equation*}
\ang{s_n(\mu),\lambda_n^\pr(\mu)} = \textstyle\sum\limits_{i=1}^{n_\sQ} \sum\limits_{j=1}^{n_\sS} \uus{}_{n,j} (\mu)\uulambda{}_{n,i}^{\pr}(\mu) \upsi^T_i \uzeta_j.
\end{equation*}
We thus compute {\it offline} the product $\upsi_i^T\uzeta_j$ at cost $\cO(n_{\sQ}n_\sS \mathcal{N}_\sQ)$; in the {\it online} stage, we simply compute the sum at cost $\cO(n_{\sQ}n_\sS)$.  

In summary, the primal-dual approach presented here computes {\it fully online-efficient} approximations and associated {\it a posteriori}  error bounds.  In comparison with the primal-only approach, the primal-dual approach has the slight disadvantage that it requires the setup (offline) and solution (online) of an additional RB approximation problem for the slack variable $s$. The payoff, however, is in the {\it a posteriori} error estimation stage: whereas the primal-only method requires the use of nonlinear projections back into the FE space, the primal-dual approach does not.  The online cost for the former thus depends on the FE dimension $\mathcal{N}$, while that of the latter depends {\it only} on the RB dimension $n_\sV$, $n_\sQ$, and $n_\sS$.   

\section{Example: The Reduced-Basis Method for the Obstacle Problem}
\label{sec:example}

In Sec.~\ref{sec:rb}, we presented the framework for the RB approximation of variational inequalities of the first kind.  We now apply methods of Secs.~\ref{sec:primal} and \ref{sec:duality} to two model problems.  Model 1 is taken from \cite{HSW2012} and represents a 1D elastic rope over a rigid obstacle.  Model 2 is a 2D extension of Model 1 and  represents an elastic membrane below a rigid obstacle.  We describe each model problem in more detail below.  

\subsection{Problem Statement}

In this section, we describe the two model problems against which we shall test the performance of our proposed approach. 

\subsubsection{Model 1} 

First, we consider a one-dimensional problem  with domain $\Omega=(0,1)$, scalar parameter domain $\mathcal{D}=[0.001,0.01]$, and $Q_a=1$. The bilinear form $a(\cdot,\cdot;\mu):\sV\times\sV\rightarrow\mathbb{R}$ and bilinear form $b(\cdot,\cdot):\sV\times\sQ\rightarrow\mathbb{R}$ are defined as (see \cite{HSW2012}): for any $\mu \in \cD$, and for all $w,v \in \sV$ and $q \in \sQ$, 
$$
a(v,w;\mu) = \mu\int_{\Omega}v_xw_x\,dx,  \qquad 
b(v,q) = -q(v)=-\ang{q,v}.
$$
Hence, it follows that $B=-I$. For $h(x)=5x-10$, the linear form $f(\cdot) \in \sV'$ and $g(\cdot)\in\sQ'$ are defined as (again, see \cite{HSW2012}): for all $v \in \sV$ and $q \in \sQ$, 
$$
f(v)= -\int_{\Omega} v(x)dx, \qquad 
g(q) = \ang{g,q}=\textstyle\sum\limits_{i=1}^{\mathcal{N}}q_ih(x_i), \mbox{ with } q=\textstyle\sum\limits_{i=1}^{\mathcal{N}}q_i\phi_i.
$$
We impose homogeneous Dirichlet conditions on both boundaries.  We note that we take the space $\sQ$ as the dual space of $\sV=H^1_0(\Omega)$. Since $\sV$ is reflexive, we have $\sV'=\sQ$, and $\sQ'=\sV$. This model represents an elastic rope with different elasticity moduli and constant body force. Solution of the variational inequality thus finds the equilibrium condition that minimizes the potential energy subject to the constraint presented by the obstacle.   A sample solution for $\mu = 0.01$ is shown in Fig.~\ref{fig:ex1}(a). 

\subsubsection{Model 2}

We now introduce a second model problem which will allow us to thoroughly examine the performance of the proposed methods as the FE dimension $\cN$ increases.  We thus extend the one-dimensional example to two dimensions, and consider a problem with domain $\Omega=(0,1)\times (0,1)$, and scalar parameter domain $\mathcal{D}=[0.45,0.55]$. The bilinear form $a$ and $b$ are defined as in Model 1: for any $\mu \in \cD$, and for all $w,v \in \sV$ and $q \in \sQ$, 
$$
a(v,w;\mu) = \mu\int_{\Omega}\nabla v \cdot \nabla w \,d\Omega, \qquad
b(v,q) = q(v)=\ang{q,v}.
$$
We then define the linear forms $f$ and $g$ as: for all $v \in \sV$ and $q \in \sQ$, 
$$
f(v) = \int_{\Omega} v \,d\Omega, \quad \forall \ v\in\sV, \qquad
g(q) = 0.1 \textstyle\sum\limits_{i=1}^{\mathcal{N}}q_i, \mbox{ with } q=\textstyle\sum\limits_{i=1}^{\mathcal{N}}q_i\phi_i.
$$
Here, we again impose homogeneous Dirichlet conditions on both boundaries and $\sQ$ is the dual space of $\sV=H^1_0(\Omega)$. Since $\sV$ is reflexive, we again have $\sV'=\sQ$, and $\sQ'=\sV$. This model represents an elastic membrane  below a rigid obstacle acted on by a constant body force. A sample solution for $\mu = 0.5$ is shown in Fig.~\ref{fig:ex1}(b).

\begin{figure}[ht]
\begin{minipage}[ht]{0.49\linewidth}
\centerline{\includegraphics[height=5cm]{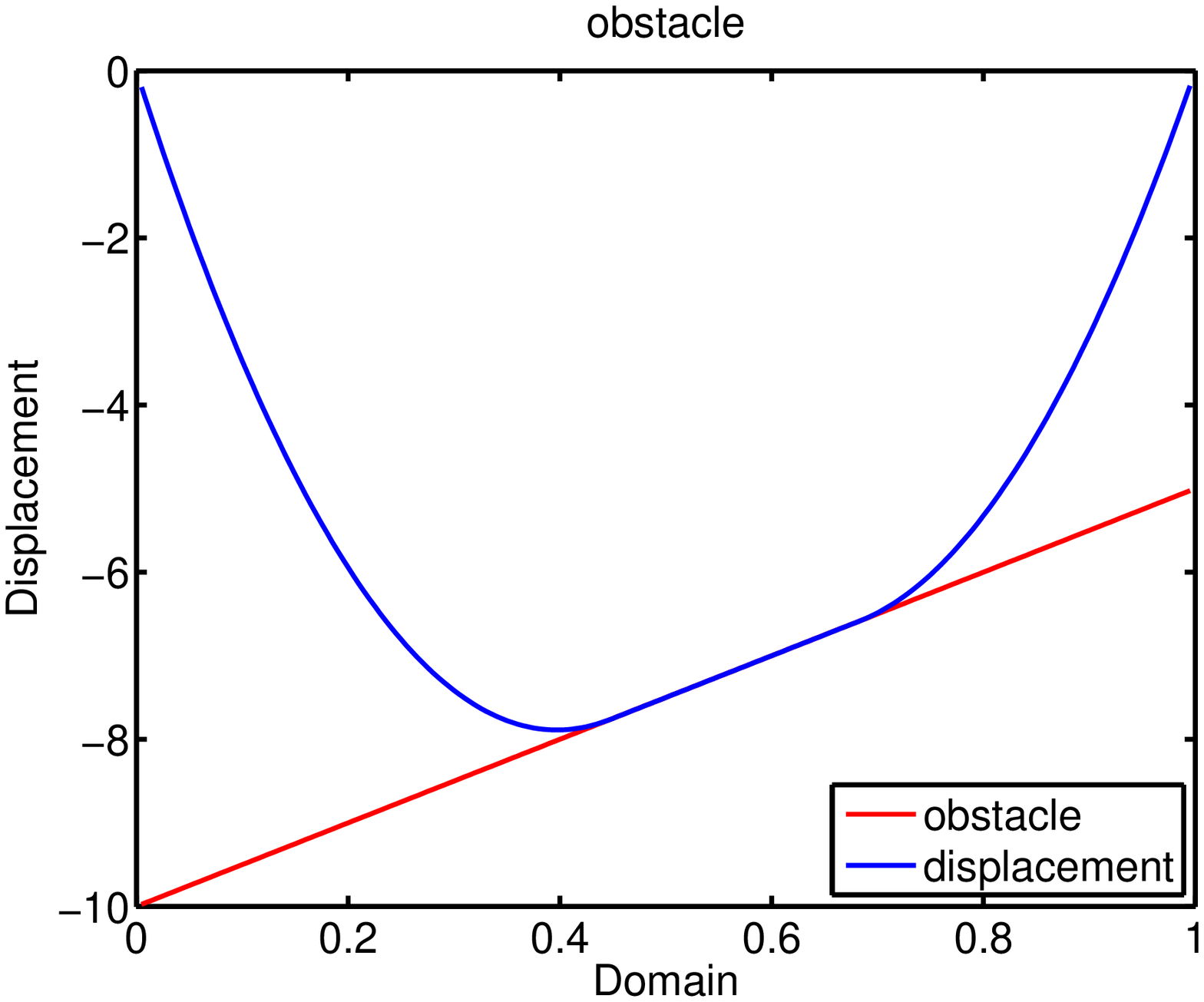}}
\centerline{\small (a)}
\end{minipage}
\begin{minipage}[ht]{0.49\linewidth}
\centerline{\includegraphics[height=5cm]{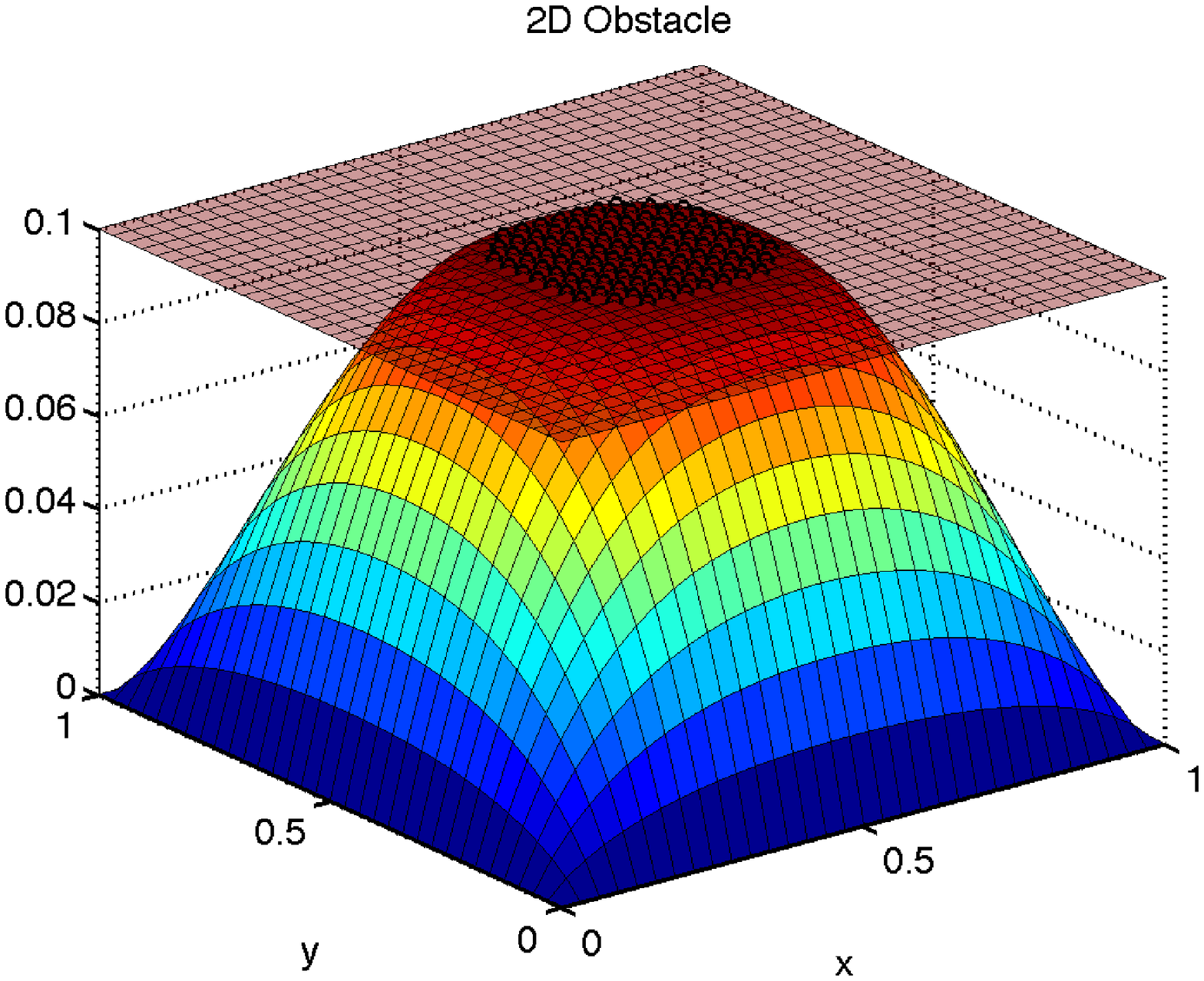}}
\centerline{\small (b)}
\end{minipage}
\caption{Sample solutions for (a) Model 1 with $\mu = 0.01$ and (b) Model 2 with $\mu = 0.5$.}
\label{fig:ex1}
\end{figure}

\subsection{Numerical Results}

In this section, we test the primal-only and primal-dual approaches described in Sections \ref{sec:primal} and \ref{sec:duality} using our two model problems.   In Model 1, we use a triangulation consisting of $200$ elements (i.e., segments) in the one-dimensional domain. In model 2, we use a triangulation consisting of $32\times 32$ elements in the two-dimensional domain. For both model problems, we use standard conforming first order nodal bases for $\sV$. For the basis functions of $\sQ$, we use the biorthogonal functions of the basis functions of $\sV$.

The numerical results for Model 1 are attained using the quadratic optimisation capabilities of M{\small ATLAB} \cite{MATLAB2013a}.  More specifically, we use interior point method through the built-in optimisation function {\tt quadprog}. The numerical results for Model 2 are obtained using the open-source software rbOOmit \cite{KP2011a}, an implementation of the RB framework within the C++ finite element  library libMesh \cite{KPSC2006}.  We now compare the performance of the two approaches presented in Sec.~\ref{sec:primal} and \ref{sec:duality}, focusing particularly on approximation accuracy, error bound sharpness, and computational efficiency.  

\subsubsection{Error bounds}

In order to reproduce the results of \cite{HSW2012} and compare therewith the performance of the proposed primal-dual approach, we follow the testing procedure described in \cite{HSW2012}.  We thus take the test sample set $\mathcal{F}$ as $250$ parameters uniformly distributed in the parameter domain, and the RB basis space as an equidistant sample of $n$ parameters from the parameter domain $\mathcal{D}$ . The RB space and test samples are constructed in the same way for Model 2.  Note that in both Models 1 and 2, $Q_a=Q_f=1$, we only need to include $A^{-1}f$ to ensure inf-sup stability (see \cite{HSW2012}). Hence, in both cases, $n_{\sup}=1$.  

We now compare the performance of the primal-only and primal-dual approaches.  We begin with the primal variable $u$ and present in Fig.~\ref{fig:num-u}(a) and (b) (for Models 1 and 2, respectively) the maximum relative error $\max_{\mu\in\mathcal{F}}(\|u(\mu)-u_n^{m}(\mu)\|_\sV/\|u(\mu)\|_\sV)$ and maximum relative error bound $\max_{\mu\in\mathcal{F}}(\Delta^{m}_n(\mu)/\|u(\mu)\|_\sV)$, for the primal-only approach (${m} = ``\pr"$) and the primal-dual approach (${m} = ``\pr,\du"$).  We note that in both model problems, the error in the RB approximation using the primal-only approach (shown using blue crosses) and the primal-dual approach (red crosses) almost coincide.  However, the primal-only error bound (blue circles)  does not replicate the convergence rate of the exact error and becomes increasingly pessimistic as $n$ increases.  The results for the primal-dual approach (red circles), on the other hand, are not only sharper, but mimics the true convergence rate of the approximation.  

We now turn to the dual variable $\lambda$.  We present in Fig.~\ref{fig:num-l}(a) and (b) (for Models 1 and 2, respectively) the maximum relative error $\max_{\mu\in\mathcal{F}}(\|\lambda(\mu)-\lambda_n^{\pr}(\mu)\|_\sQ/\|\lambda(\mu)\|_\sQ)$ and maximum relative error bound $\max_{\mu\in\mathcal{F}}(\Delta^{m}_n(\mu)/\|\lambda(\mu)\|_\sQ)$, for the primal-only approach (${m} = ``\pr"$) and the primal-dual approach (${m} = ``\pr,\du"$).  We recall that in both approaches, the same approximation, $\lambda_n^\pr(\mu)$ is used for the dual variable.  Thus, only one error curve (in black) is shown for both approaches.

Once again, the primal-only error bound (blue circles) is more pessimistic than the primal-dual error bound (red circles).  This behavior is consistent with the results for the primal variable, especially since the error bounds $\Delta_\lambda^\pr$ and $\Delta_\lambda^{\pr,\du}$ (see (\ref{eq:pr-bound-l}) and (\ref{eq:du-bound-l})) contain the error bound for the primal variable $\Delta_u^\pr$, $\Delta_u^{\pr,\du}$, respectively.  

The greatly improved performance (i.e. sharpness) of the primal-dual error bound can be attributed to the observations made at the ends of Sections \ref{sec:pr-a-post-error} and \ref{sec:dual-approximation}.  The strictly feasible approximations $u_n^\du$ enabled the computation of considerably sharper error bounds without necessitating the use of a nonlinear projection $\Pi$ and at the cost only of an additional RB approximation problem.   However, the latter two considerations require a careful comparison of the computational cost of both approaches.  We thus consider next the online efficiency of the two methods.  

\begin{figure}[ht]
\begin{minipage}[ht]{0.49\linewidth}
\centerline{\includegraphics[width=0.95\textwidth]{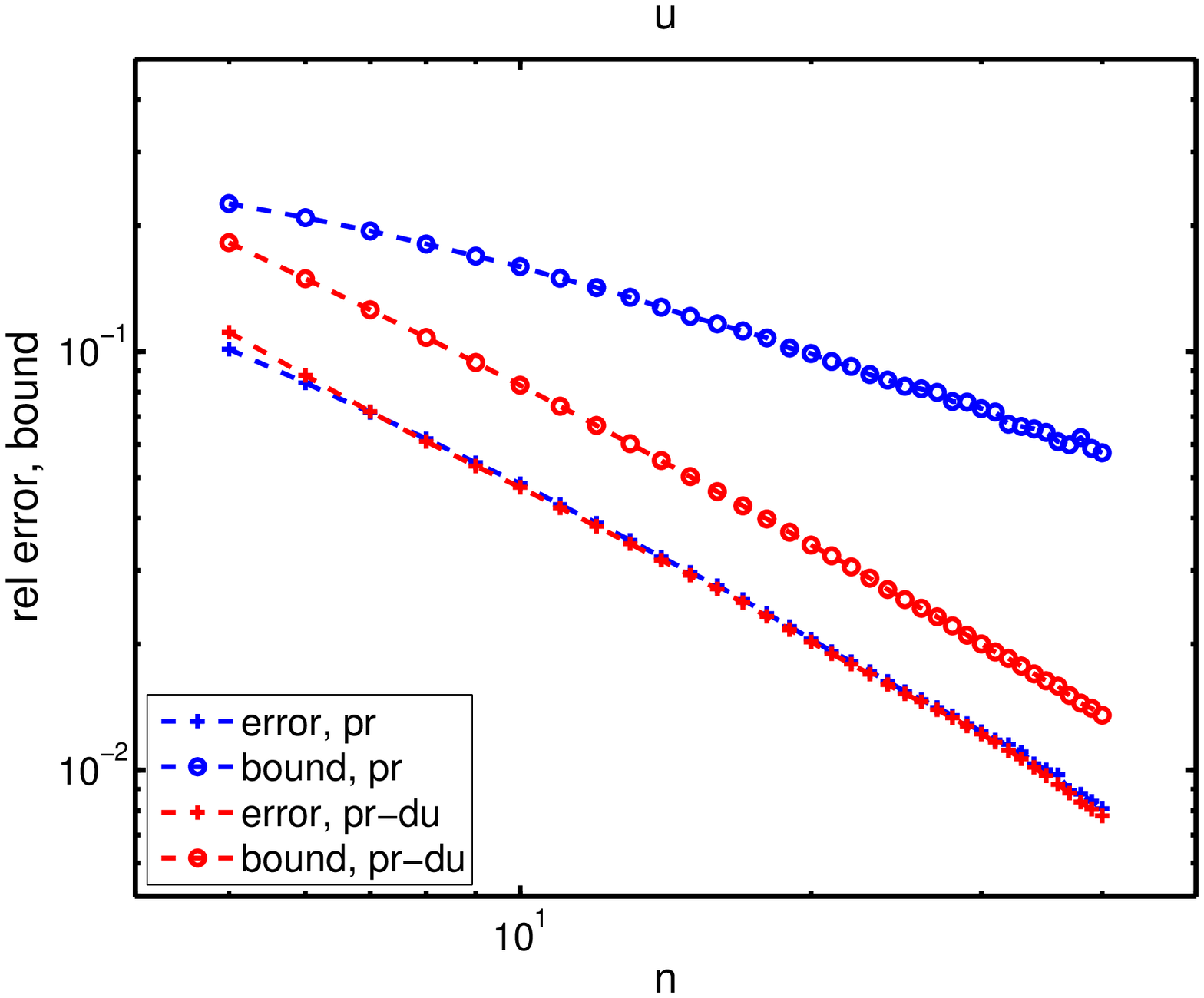}}
\vspace{-0.1cm} 
\centerline{\small (a) Model 1}
\end{minipage}
\begin{minipage}[ht]{0.49\linewidth}
\centerline{\includegraphics[width=0.95\textwidth]{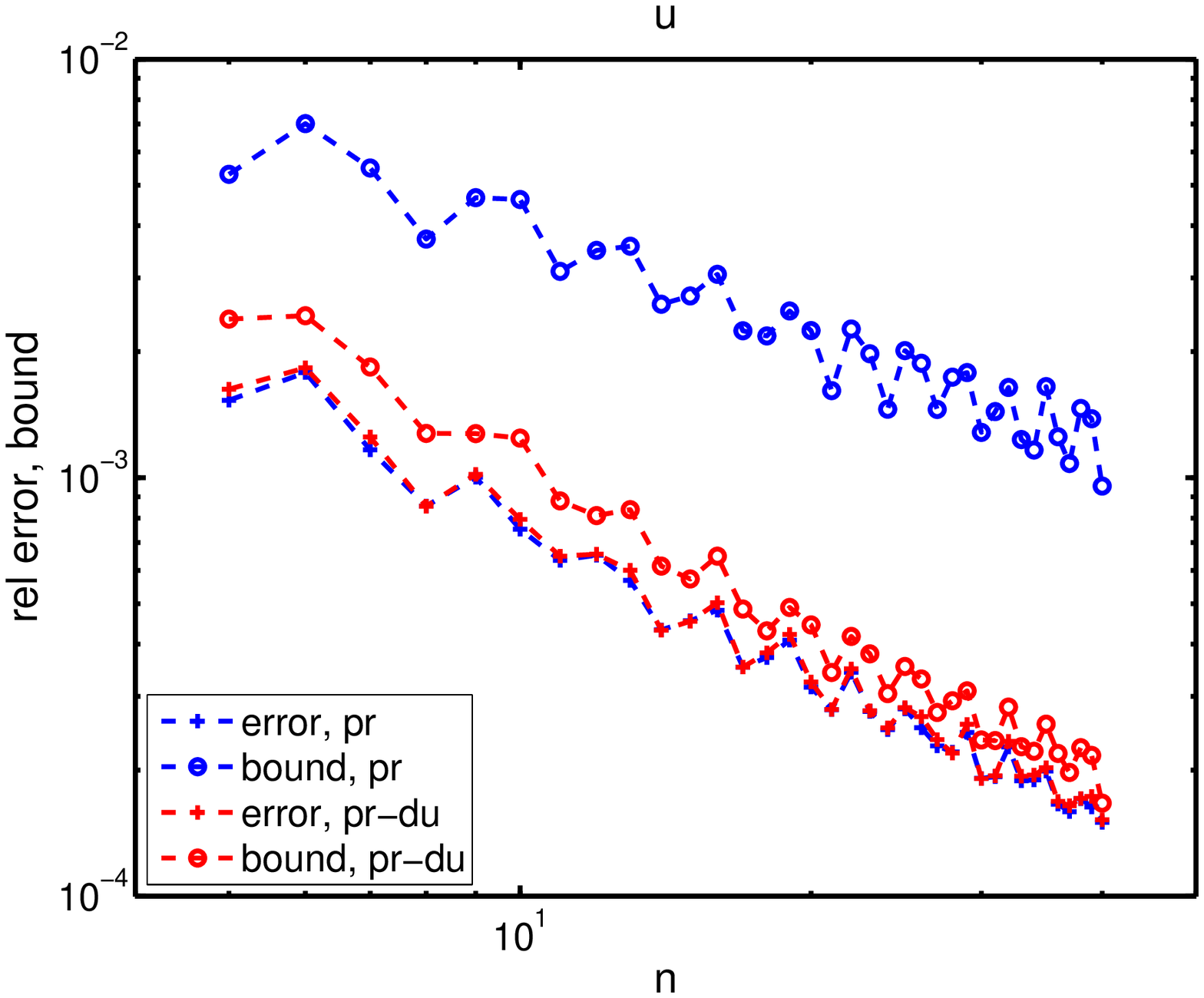}}
\vspace{-0.1cm} 
\centerline{\small (b) Model 2}
\end{minipage}
\caption{Maximum relative error and {\it a posteriori} error bound  for $u$ obtained using the primal-only approach and primal-dual approach.}
\label{fig:num-u}
\end{figure}
\begin{figure}[ht]
\begin{minipage}[ht]{0.49\linewidth}
\centerline{\includegraphics[width=0.95\textwidth]{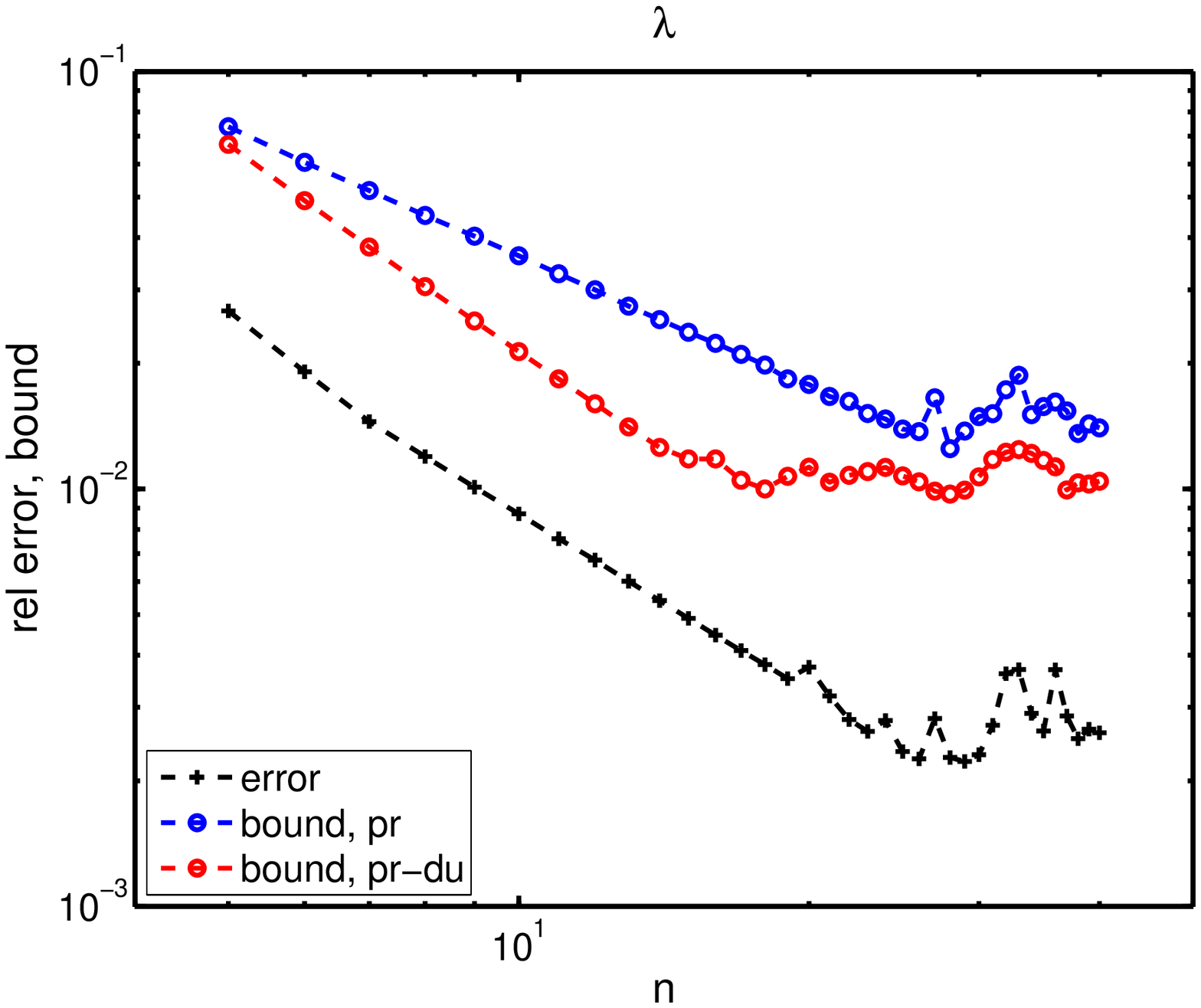}}
\vspace{-0.1cm} 
\centerline{\small (a) Model 1}
\end{minipage}
\begin{minipage}[ht]{0.49\linewidth}
\centerline{\includegraphics[width=0.95\textwidth]{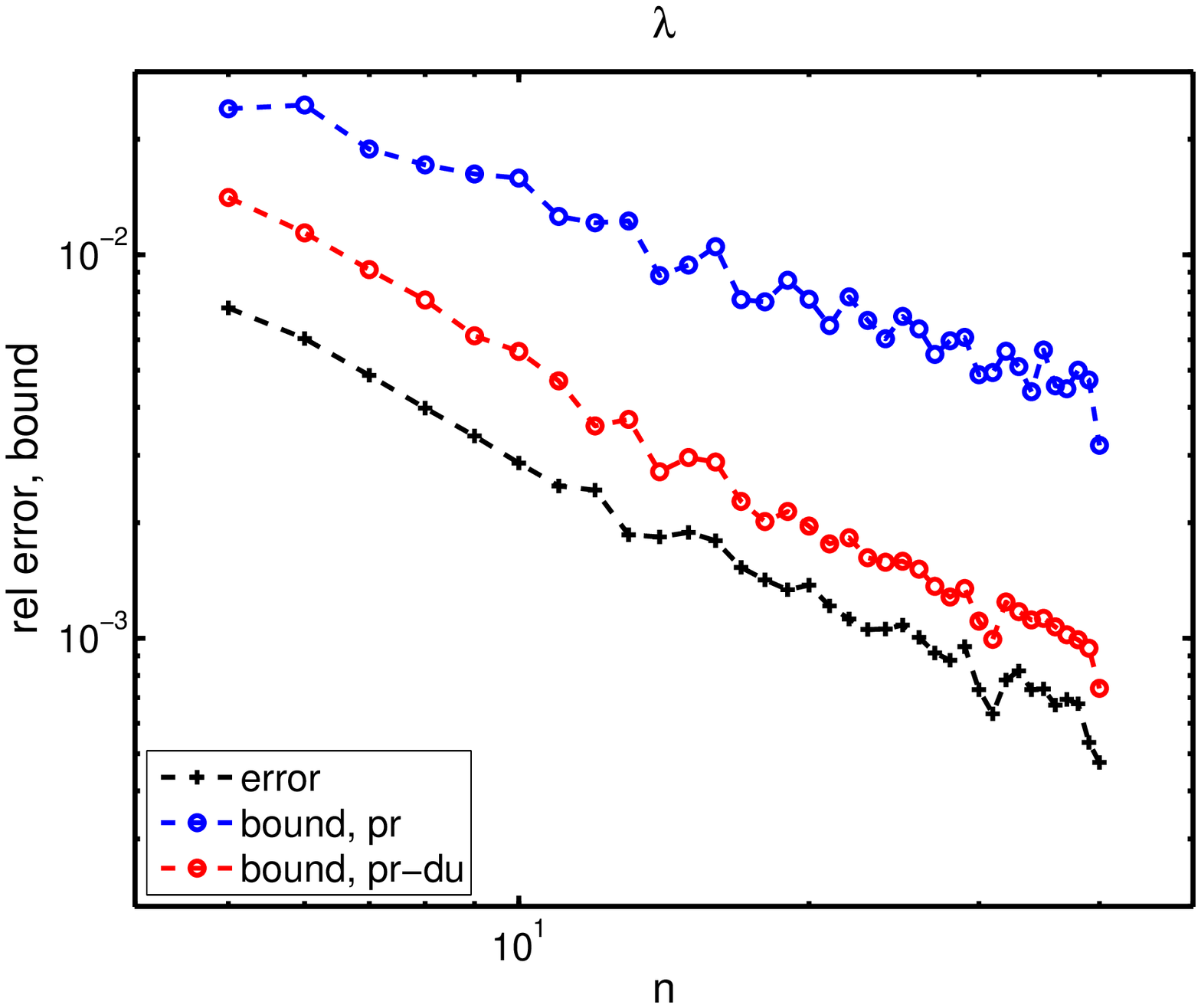}}
\vspace{-0.1cm} 
\centerline{\small (b) Model 2}
\end{minipage}
\caption{Maximum relative error and {\it a posteriori} error bound  for $\lambda$ obtained using the primal-only approach and primal-dual approach.}
\label{fig:num-l}
\end{figure}

\subsubsection{Online efficiency}

We present in Fig.~\ref{fig:num-time} the maximum relative error bound for the two approaches plotted against the average total online computational time for a single evaluation of the approximation and error bound.  Figs.~\ref{fig:num-time}(a) and (b) show results for Model 1, and we observe that for a given (commonly attainable) error, the primal-dual approach entails a higher online computational cost than the primal-only approach.  This is due to the additional reduced-basis problem required by the primal-dual approach, the cost of which does not offer any computational savings for the case when $\cN$ is small (as is the case in this simple one-dimensional problem).  

In contrast, the results for Model 2 show that for a given accuracy, the online cost for the primal-dual approach is lower than that of the primal-only approach.  In this example, the $\cN$-dependent online cost for the primal-only approach is high enough so as to justify the additional cost of the dual reduced problem.  We can thus reasonably expect that computational advantages of the primal-dual approach will become even more pronounced as $\cN$ increases, for example, in the case of three-dimensional problems.  Furthermore, it can be clearly seen in Figs.~\ref{fig:num-time}(c) and (d) (and partly in Fig.~\ref{fig:num-time}(a)) that the superior sharpness of the primal-dual error bounds allow us to achieve greater accuracy than in the primal-only approach.  

\begin{figure}[ht]
\begin{minipage}[ht]{0.49\linewidth}
\centerline{\includegraphics[width=0.95\textwidth]{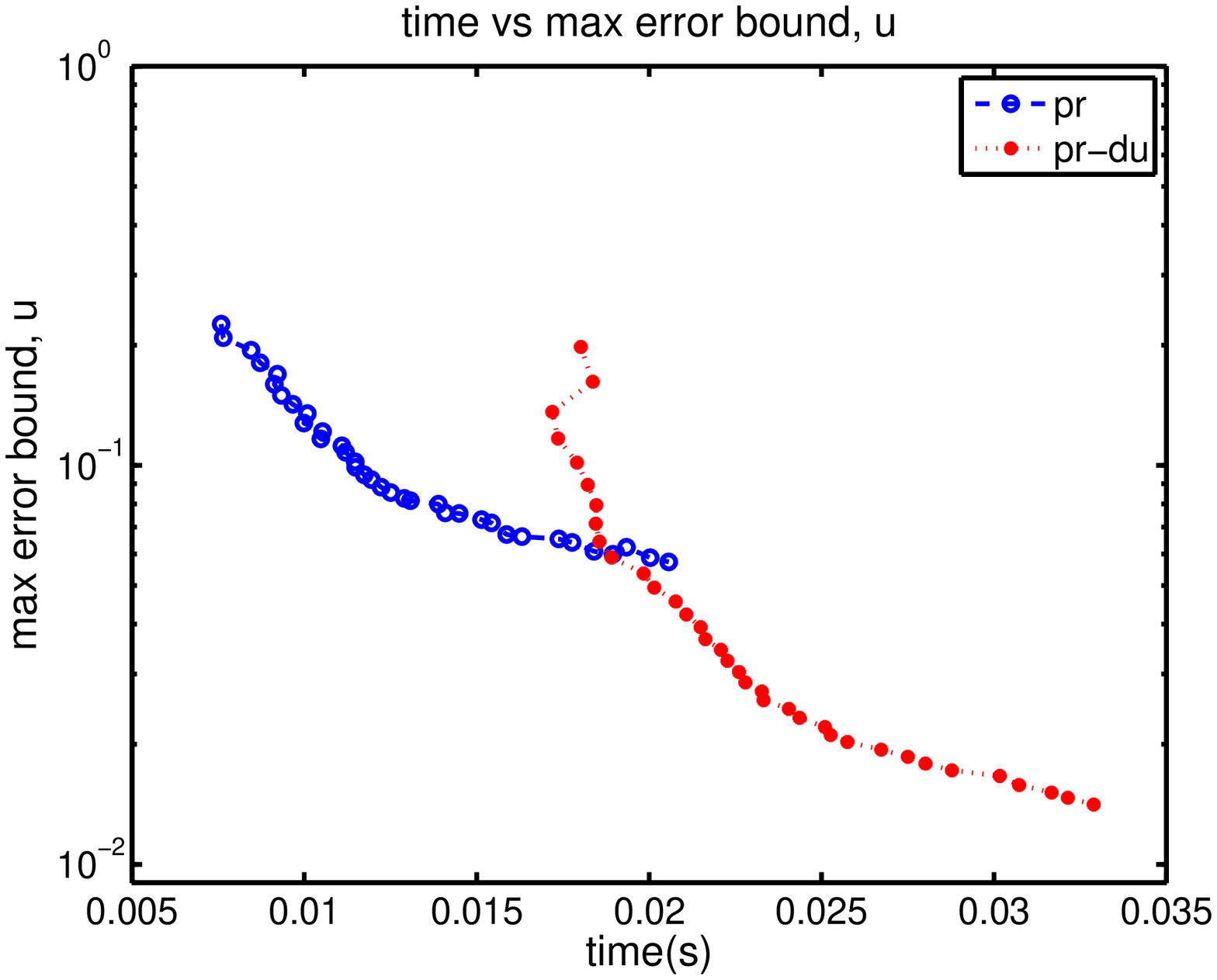}}
\vspace{-0.1cm} 
\centerline{\small (a)}
\end{minipage}
\begin{minipage}[ht]{0.49\linewidth}
 \centerline{\includegraphics[width=0.95\textwidth]{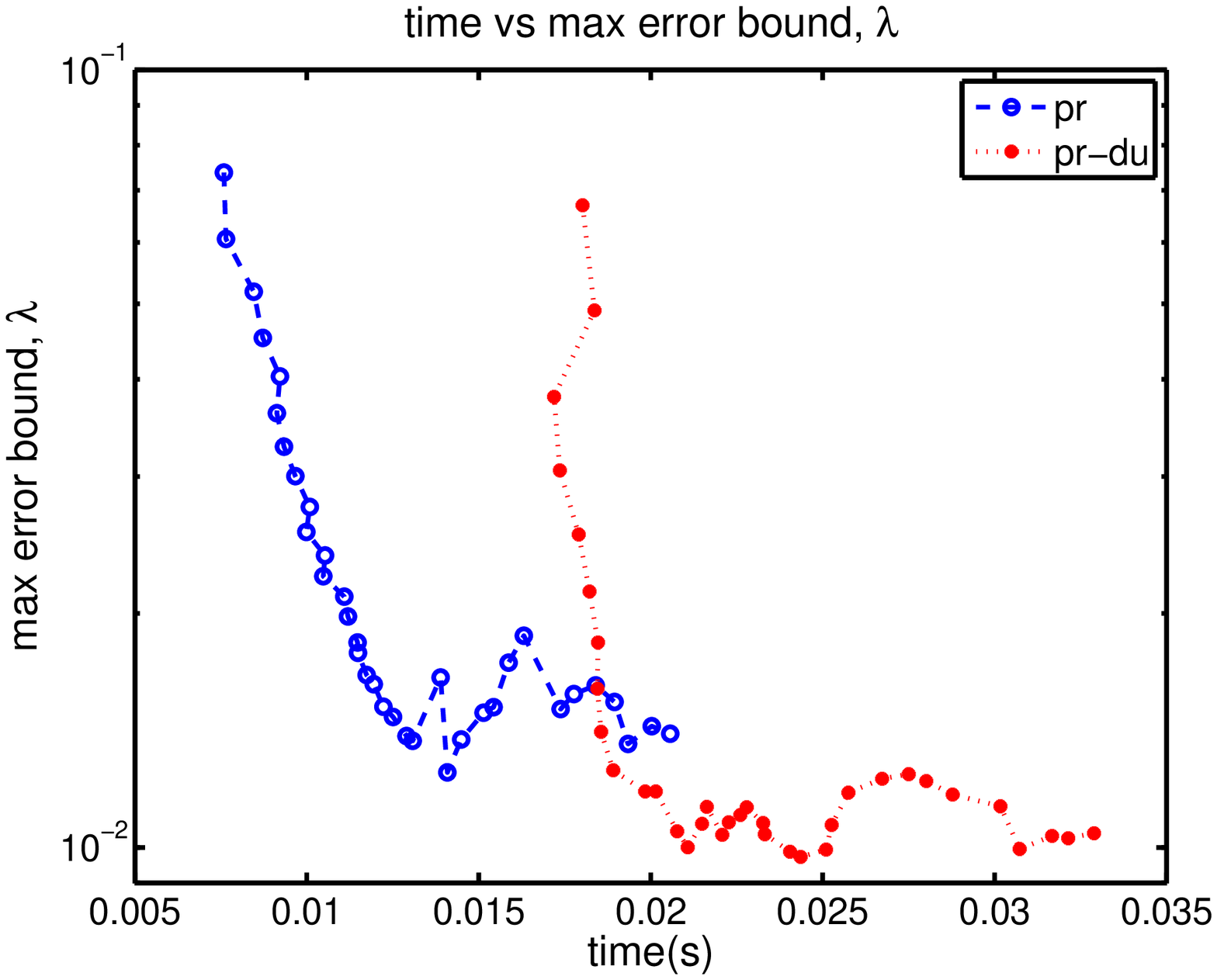}}
\vspace{-0.1cm} 
 \centerline{\small (b)}
\end{minipage}
\begin{minipage}[ht]{0.49\linewidth}
\centerline{\includegraphics[width=0.95\textwidth]{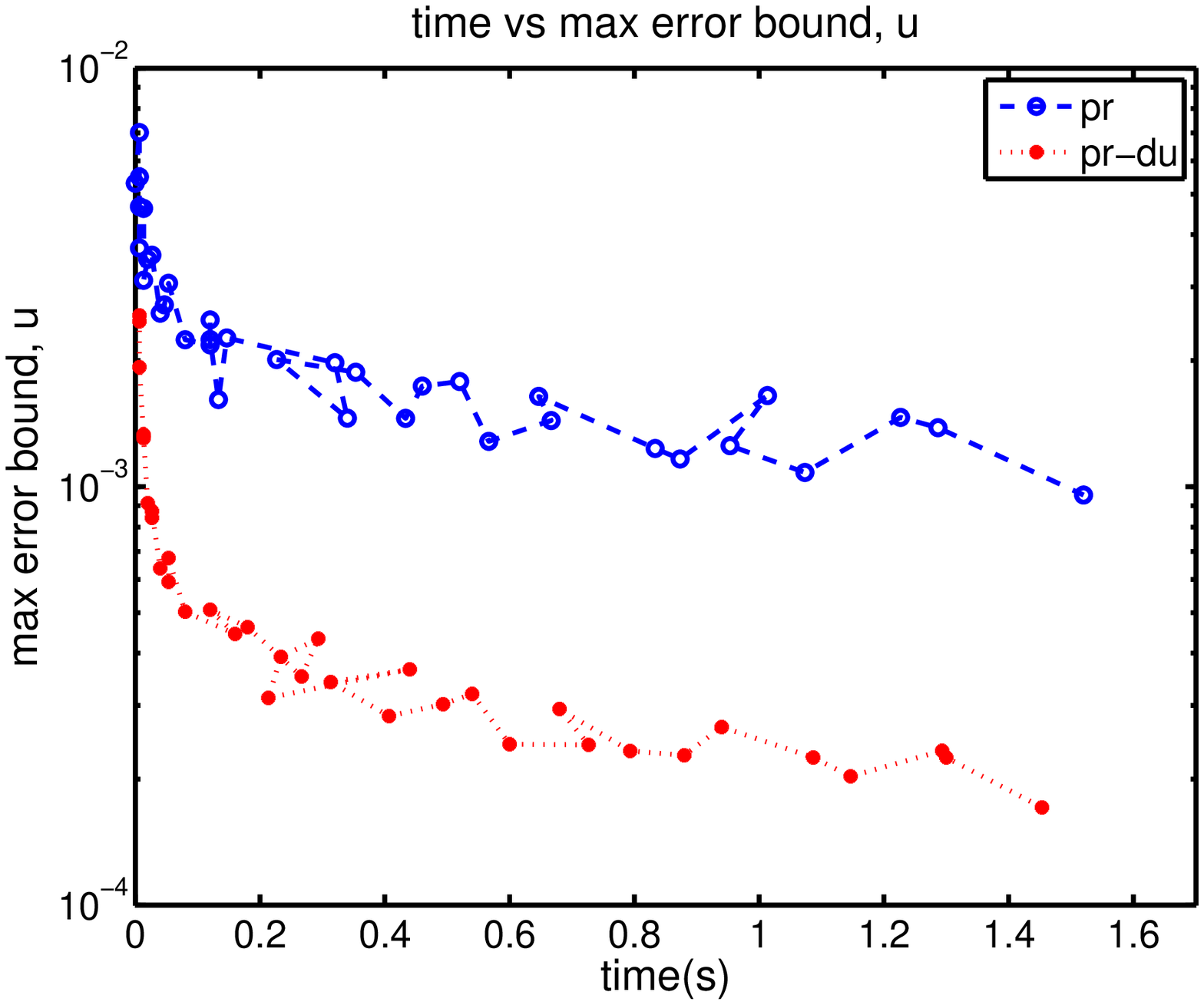}}
\vspace{-0.1cm} 
\centerline{\small (c)}
\end{minipage}
\begin{minipage}[ht]{0.49\linewidth}
 \centerline{\includegraphics[width=0.95\textwidth]{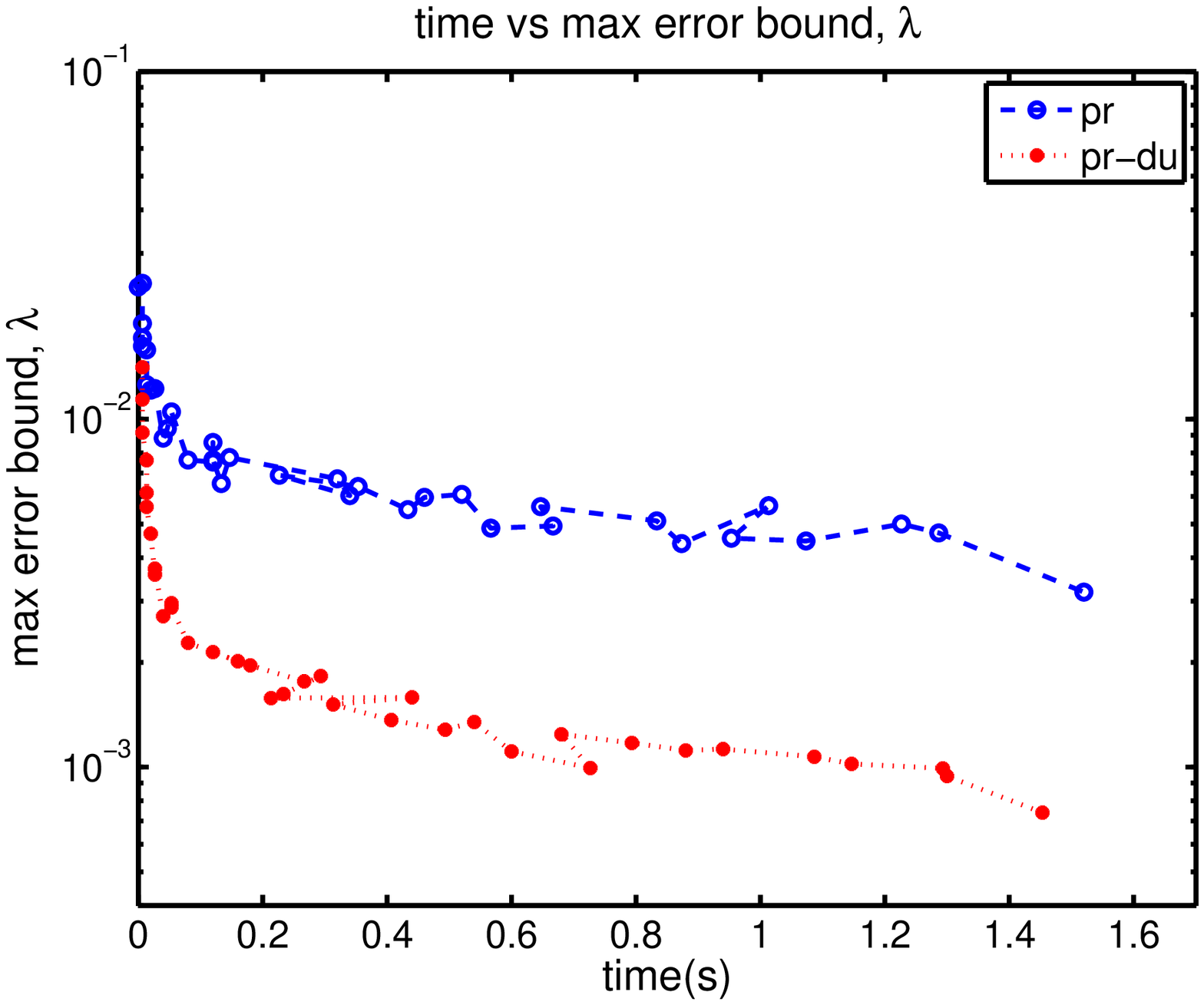}}
\vspace{-0.1cm} 
 \centerline{\small (d)}
\end{minipage}
\caption{Comparison of online computational times for (a) $u$ and (b) $\lambda$ for Model 1, and (c) $u$ and (d) $\lambda$ for Model 2.}
\label{fig:num-time}
\end{figure}

\section{Summary and Perspectives}
We proposed a primal-dual approach for computing approximations and associated {\it a posteriori} error bounds to solutions of variational inequalities of the first kind.  The proposed approach utilizes an additional approximation problem for the slack variable in order to obtain strictly feasible (primal-dual) approximations.  This in turn enables the derivation of sharp {\it a posteriori} error bounds which closely mimic the convergence rate of the corresponding approximation.  Applied to the reduced basis method, the approach further allows a {\it full} offline-online computational decomposition in which the online cost to compute the error bounds is completely independent of the dimension $\cN$ of the full problem.  Numerical results illustrate the superiority of the approach in cases where the dimension $\cN$ of the full problem is high.  Future work will focus on ({\it i\/}) the application of the method to more complex problems, particularly to elastic contact, and ({\it ii\/}) the development of appropriate greedy strategies for the systematic selection of basis functions.  

\section*{Acknowledgments} We would like to thank Prof.~Michael Herty and Mark K\"archer of RWTH Aachen University for the helpful discussions and comments.  

\appendix
\bibliography{All_References}
\bibliographystyle{siam}

\end{document}